\documentclass[11pt,leqno]{article}     
\usepackage[utf8]{inputenc}
\usepackage{amssymb,amsmath,latexsym,mathtools}
\usepackage[english]{babel}
\usepackage{bbm}
\usepackage{hyperref}
\usepackage{parskip}
\usepackage{enumitem}
\numberwithin{equation}{section}  
\usepackage{amsthm}     
\newtheorem{theorem}{Theorem}
\newtheorem*{theorem-non}{Theorem}
\newtheorem{corollary}[theorem]{Corollary}
\newtheorem{lemma}[theorem]{Lemma}
\newtheorem{proposition}[theorem]{Proposition}
\newtheorem*{proposition-non}{Proposition}
\numberwithin{theorem}{section}
\theoremstyle{definition}
\newtheorem{definition}[theorem]{Definition}
\newtheorem{example}[theorem]{Example}

\theoremstyle{remark}
\newtheorem{remark}[theorem]{Remark}
\usepackage{graphicx} 
\usepackage{tikz-cd}
\DeclareMathOperator{\sgn}{sgn}

\begin{document}
\title{Rationality of the periods of Eisenstein series}
\author{Soumyadip Sahu\\School of Mathematics, Tata Institute of Fundamental Research,\\ 1 Homi Bhabha Road, Mumbai, 400005, Maharashtra, India,\\Email: soumyadip.sahu00@gmail.com}
\date{}
	
\maketitle
\begin{abstract}
The article generalizes an observation of Zagier and Gangl to show that the image of the spectral Eisenstein series on a general congruence subgroup of $\text{SL}_2(\mathbb{Z})$, under the Eichler-Shimura isomorphism, is defined over a cyclotomic number field. We use the same technique to generalize an invariant attached to imaginary quadratic fields in connection with the polylogarithm conjecture on the special values of $L$-functions. Our treatment also provides an elementary derivation of the Fourier expansion of the Maass Eisenstein series on congruence subgroups presented as a power series.   
\end{abstract}
 


\section{Introduction}
\label{section1}
Let $k \geq 2$ be an integer and $V_{k-2, \mathbb{C}} \subseteq \mathbb{C}[X]$ be the $\mathbb{C}$-subspace of polynomials having degree $\leq k-2$. There is a right $\mathbb{C}[\Gamma]$-module stricture on $V_{k-2,\mathbb{C}}$ described by 
\[P(X)\lvert_{\gamma} = (cX+d)^{k-2}P(\frac{aX+b}{cX+d}). \hspace{.3cm}(\substack{\gamma \in \Gamma})\]
Let $\mathcal{G}$ be a congruence subgroup of $\text{SL}_2(\mathbb{Z})$. Suppose that 
$\mathcal{M}_{k}(\mathcal{G})$, resp. $\mathcal{S}_{k}(\mathcal{G})$, be the $\mathbb{C}$-vector space of modular forms, resp. cusp forms, on $\mathcal{G}$. Now suppose $f \in \mathcal{M}_k(\mathcal{G})$. With each $\tau_0$ in the upper half-plane we associate a cocyle on $\mathcal{G}$ with values in $V_{k-2, \mathbb{C}}$ as follows:  
\begin{equation}
\label{1.1}
\text{ES}_{k}^{\tau_0}(f) : \mathcal{G} \to V_{k-2, \mathbb{C}}, \hspace{.3cm} \gamma \mapsto \int^{\tau_0}_{\gamma^{-1}\tau_0}f(\tau)(X-\tau)^{k-2}d\tau.
\end{equation}
The image of $\text{ES}_{k}^{\tau_0}(f)$ in $H^{1}(\mathcal{G}, V_{k-2, \mathbb{C}})$ depends only $f$ and is independent of the choice of $\tau_0$. One defines the Eichler-Shimura homomorphism 
\[[\text{ES}_k]: \mathcal{M}_{k}(\mathcal{G}) \to H^{1}(\mathcal{G}, V_{k-2,\mathbb{C}})\]
by setting $[\text{ES}_k](f) = [\text{ES}_{k}^{\tau_0}(f)]$ for some $\tau_{0}$. Suppose that $\mathcal{S}^{a}_k(\mathcal{G}) := \{\bar{f} \mid f \in \mathcal{S}_k(\mathcal{G})\}$ is the space of antiholomorphic cusp forms. Define the antiholomorphic Eichler-Shimura homomorphism as 
\[[\text{ES}_k]^{a}: \mathcal{S}^{a}_{k}(\mathcal{G}) \to H^{1}(\mathcal{G}, V_{k-2,\mathbb{C}}), \hspace{.2cm} \bar{f} \mapsto \#[\text{ES}_k](\bar{f})\]
where $\#$ refers to the complex conjugation on $H^{1}(\mathcal{G}, V_{k-2,\mathbb{C}})$ induced by the complex conjugation on $V_{k-2,\mathbb{C}}$. The Eichler-Shimura theorem for modular forms states that the Eichler-Shimura homomorphisms together give rise to a Hecke-equivariant isomorphism 
\begin{equation}
\label{1.2}
[\text{ES}_k] \oplus [\text{ES}_k]^{a}: \mathcal{M}_{k}(\mathcal{G}) \oplus \mathcal{S}^{a}_{k}(\mathcal{G}) \xrightarrow{\cong} H^{1}(\mathcal{G}, V_{k-2,\mathbb{C}});
\end{equation}
see \cite{knopp} and Remark~\ref{remark2.3}. The space of Eisenstein series on $\mathcal{G}$, denoted $\mathcal{E}_k(\mathcal{G})$, is the unique linear complement of $\mathcal{S}_k(\mathcal{G})$ in $\mathcal{M}_k(\mathcal{G})$ that is orthogonal to $\mathcal{S}_{k}(\mathcal{G})$ with respect to the Peterson inner product. One can explicitly construct $\mathcal{E}_{k}(\mathcal{G})$ using the spectral Eisenstein series on $\mathcal{G}$. Suppose that $-\text{Id} \notin \mathcal{G}$ whenever $k$ is odd. 
Recall that the spectral Eisenstein series are parameterized by the cusps of $\mathcal{G}$ if $k$ is even and the regular cusps of $\mathcal{G}$ if $k$ is odd. For a cusp $x$ of $\mathcal{G}$, let $E_{k,x}$, whenever defined, denote the spectral Eisenstein series attached to $x$. The collection of functions 
\[\mathcal{B}_{\mathcal{G},k} := \{E_{k,x} \mid \text{$x$ is regular if $k$ is odd}\} \hspace{.3cm}\substack{(k \geq 3)}\]  
forms a basis of $\mathcal{E}_k(\mathcal{G})$ for each $k \geq 3$. If $k = 2$, then the spectral Eisenstein series are nonholomorphic due to a term arising from analytic continuation. Nevertheless, there are explicit rational numbers $\alpha_x \in \mathbb{Q}^{\times}$ so that   
\[\mathcal{B}_{\mathcal{G},2} :=\{E_{2,x} -  \frac{\alpha_x}{\alpha_{\infty}}E_{2,\infty}\mid x \neq \mathcal{G}\infty\}\]
is a basis of $\mathcal{E}_2(\mathcal{G})$ (Section~\ref{section3.1}). For convenience, write $\mathcal{B}_{\mathcal{G},k} = \emptyset$ if $k$ is odd and $-\text{Id} \in \mathcal{G}$. The primary goal of this article is to study the image of $\mathcal{B}_{\mathcal{G},k}$ under the Eichler-Shimura isomorphism \eqref{1.2}. Suppose that $\mathbb{K}$ is a subfield of $\mathbb{C}$. Then the inclusion $V_{k-2, \mathbb{K}} \hookrightarrow V_{k-2, \mathbb{C}}$ induces an change of coefficients isomorphism $H^{1}(\mathcal{G}, V_{k-2,\mathbb{K}}) \otimes \mathbb{C} \cong H^{1}(\mathcal{G}, V_{k-2,\mathbb{C}})$. We say, a class $\alpha \in  H^{1}(\mathcal{G}, V_{k-2,\mathbb{C}})$ is defined over $\mathbb{K}$ if $\alpha$ lies inside the canonical $\mathbb{K}$-form $H^{1}(\mathcal{G}, V_{k-2,\mathbb{K}})$, i.e., $\alpha$ admits a cocyle representative with values in $V_{k-2,\mathbb{K}}$. In the theory of automorphic forms \cite{harder}, one constructs a rational structure on the Eisenstein subspace of cohomology by pulling back the rational structure of the boundary. Our first result is a modular form version of the automorphic construction that works for all congruence subgroups, including those with torsion.     

\begin{theorem}
\label{theorem1.1}
Let $k$ be an integer $\geq 2$ and $\mathcal{G}$ be a congruence subgroup containing the principal level $\Gamma(N)$. Then the image of $\mathcal{B}_{\mathcal{G},k}$ under the Eichler-Shimura isomorphism \eqref{1.2} is defined over $\mathbb{Q}(\mu_N)$ where $\mu_N = \exp (\frac{2\pi i}{N})$. 
\end{theorem}

Proof of Theorem~\ref{theorem1.1} appears in Section~\ref{section4.2} of the article. Note that it is difficult to calculate a cocycle on $\mathcal{G}$ since this exercise requires knowledge of the function on a large set of generators. Thus, following Pasol and Popa \cite{pasol-popa}, one uses Shapiro's lemma to convert the question into a problem regarding $\text{SL}_2(\mathbb{Z})$. Section~\ref{section2} of the article is devoted to describing a robust reconstruction of the usual formalism of Eichler-Shimura theory, which allows us to connect the classical period cocycle \eqref{1.1} with the work of Pasol and Popa in a precise manner. In light of this description, Theorem~\ref{theorem1.1} seems to contradict the fact that the coefficients of the period polynomial of an Eisenstein series involve transcendental expressions related to the zeta values \cite{zagier}. However, an observation of Zagier and Gangl \cite[7]{zagier-gangl} shows that one can modify the period cocycle attached to the Eisenstein series on $\text{SL}_2(\mathbb{Z})$ by an appropriate coboundary so that the modified cocycle has coefficients in $\mathbb{Q}$. We first establish a generalization of their observation for $\Gamma(N)$ (Section~\ref{section4}) using a geometric family of Eisenstein series arising from the $N$-torsion points of the universal elliptic curve whose image under the Eichler-Shimura isomorphism is defined over $\mathbb{Q}$; cf. \cite{banerjee-merel}.  The result for the general congruence subgroup follows from the abstract equivariance properties of the Eichler-Shimura isomorphism.    

The polylogarithm conjecture \cite{zagier-gangl} is a statement about special values of certain $L$-functions that proposes to describe the $L$-values in terms of the polylogarithm functions. These instances are expected to form part of a more general theory in which a special $L$-value of `motivic origin' is expressed in terms of some transcendental regulator function. Our work concerns a refinement of the polylogarithm conjecture for number fields that predicts that the special values of the partial zeta function of an imaginary quadratic field are given by rational linear combinations of polylogarithm functions evaluated at the elements of the Hilbert class field. Now, the real valued polylogarithm function on the $m$-th Bloch group lifts to a well-defined function with values in $\mathbb{C}/\mathbb{Q}(m)$ where $\mathbb{Q}(m) := (2\pi i)^m\mathbb{Q} \subseteq \mathbb{C}$. Existence of such a lift reflects the compatibility between the monodromy action arising from the choice of branch for polylogarithms and the motivic Hodge structure. Inspired by the lifting property of polylogarithms and the conjecture for imaginary quadratic fields, Zagier and Gangl \cite[7]{zagier-gangl} constructed a naturally defined invariant lifting the value of the partial zeta functions at integers $\geq 2$. We use our computations for $\Gamma(N)$ to extend their construction to the Hecke $L$-function (Theorem~\ref{theorem5.1}). For this purpose, we also provide a derivation of the Fourier expansion of the Maass Eisenstein series with integer weight on $\Gamma(N)$ (Theorem~\ref{theorem3.4}) based on contour integration so that the Fourier series is presented as an element of \[\mathbb{C}[[q_N, \bar{q}_N]][v(\tau), \frac{1}{v(\tau)}]\]
where $v(\tau) = \text{Im}(\tau)$ and $q_N = \exp(\frac{2 \pi i\tau}{N})$; cf. \cite{o'sullivan}.
 
\subsection{Notation and conventions}    
\label{section1.1}
Throughout the article, $\xi$ and $\tau$ are variables with values in the upper half-plane $\mathbb{H}$. Let $i$ denote the imaginary unit. For $\tau \in \mathbb{H}$ one writes $\tau = u(\tau) + i v(\tau)$ where $u(\tau)$ and $v(\tau)$ are real numbers with $v(\tau) > 0$. Now suppose $z \in \mathbb{C}$ and $n$ is a positive integer. Then $\mathfrak{R}_n(z)$ refers to the imaginary or real part of $z$ according as $n$ is even or odd. Moreover, if $z_1$ and $z_2$ are two complex numbers, then $z_1 \sim_{\mathbb{Q}^{\times}} z_2$ means there exists $\alpha \in \mathbb{Q}^{\times}$ so that $z_1 = \alpha z_2$. We normalize the complex exponential function as $\textbf{e}(z):= \exp(2\pi iz)$. Our discussion also makes use of the formal binomial coefficients: 
\[\binom{X}{n} = \begin{cases}
	\frac{X(X-1)\cdots (X-n+1)}{n!}, & \text{if $n \geq 1$}\\
	1, & \text{if $n = 0$}\\
	0, & \text{if $n < 0$}
\end{cases} \hspace{.1cm} \in \mathbb{Q}[X].\]
Given a set $A$, let  $\mathbbm{1}_{A}$ denote the delta function on $A$, i.e., 
\[\mathbbm{1}_{A}(x, y) := \begin{cases}
	1, & x = y;\\
	0, & x \neq y. 
\end{cases}\]
We employ $\mathbb{K}$ to refer to a field of characteristic $0$ which in some places is assumed to be a subfield of $\mathbb{C}$. For a positive integer $N$ set $\mu_N := \mathbf{e}(\frac{1}{N}) \in \mathbb{C}^{\times}$. Let $\text{Hol}(\mathbb{H})$ denote the space of all $\mathbb{C}$-valued holomorphic functions, resp. holomorphic differential $1$-forms, on $\mathbb{H}$. During the calculations, one considers integrals along paths on $\mathbb{H}$. All paths are piecewise smooth unless otherwise stated.     

We abbreviate $\text{SL}_2(\mathbb{Z})$ as $\Gamma$ and write $T = \begin{pmatrix}
	1 & 1\\
	0 & 1
\end{pmatrix}$ and $S = \begin{pmatrix}
	0 & -1\\
	1 & 0
\end{pmatrix}$ for the generators of this group. The letter $\gamma$ stands for an arbitrary $2 \times 2$ matrix and we put $\gamma = \begin{pmatrix} 
	a & b \\
	c & d
\end{pmatrix}$ where $a,b, c,d$ take values in $\mathbb{Z}$. The group $\Gamma$ acts on $\mathbb{H}$ from left by holomorphic automorphisms via fractional linear transformations $(\gamma, \tau) \to \gamma\tau: = \frac {a\tau +b} {c\tau +d}$. Note that the fractional linear transformations define a holomorphic left $\Gamma$-action on the Riemann sphere $\widehat{\mathbb{C}} = \mathbb{C} \cup \{\infty\}$. Suppose that $\partial{\mathbb{H}}$ is the boundary of $\mathbb{H}$ in $\widehat{\mathbb{C}}$. Then $\partial{\mathbb{H}} = \{\infty\} \cup \mathbb{R} = \mathbb{P}^{1}(\mathbb{R}) \subseteq \widehat{\mathbb{C}}$. The subset $\mathbb{P}^{1}(\mathbb{Q}) \subseteq \partial\mathbb{H}$ plays a more significant role in the arithmetic setting. We define the \textit{extended} upper-half plane as \[\mathbb{H}^{\ast} := \mathbb{H} \cup \mathbb{P}^{1}(\mathbb{Q}).\] For $x \in \mathbb{P}^1(\mathbb{Q})$ let $\Gamma_{x}$ denote the stabilizer of $x$ in $\Gamma$. With this notation $\Gamma_{\infty} = \{\pm T^{n} \mid n \in \mathbb{Z}\}$. If $\mathcal{G}$ is a subgroup of $\Gamma$ then set $\mathcal{G}_x:= \mathcal{G} \cap \Gamma_x$. We employ the notation $\mathcal{C}(\cdot)$, resp. $\mathcal{C}_{\infty}(\cdot)$, for the set of cusps, resp. regular cusps, attached to a congruence subgroup. As per standard convention, one always denotes a cusp using a choice of representative in $\mathbb{P}^{1}(\mathbb{Q})$. For each $k \in \mathbb{Z}$, there is a right action $\lvert_{k}$ of $\Gamma$ on $\text{Hol}(\mathbb{H})$ given by 
\[f\lvert_{k,\gamma}(\tau):= (c\tau+d)^{-k}f(\gamma\tau).\] 
If $k$ is fixed throughout the discussion, then we drop it from the subscript. Now suppose $k \geq 2$. One uses the notation $\mathcal{M}_k(\mathcal{G})$, $\mathcal{S}_k(\mathcal{G})$, and $\mathcal{E}_k(\mathcal{G})$ to refer to the spaces of all holomorphic modular forms, cusp forms, and Eisenstein series as in the introductory discussion. Moreover, $\mathcal{S}^{a}_k(\mathcal{G}) = \{\bar{f} \mid f \in \mathcal{S}_k(\mathcal{G})\}$ is the space of antiholomorphic forms.

Let $n$ be a nonnegative integer. Set $V_{n}:= \{P \in \mathbb{Z}[X] \mid \text{deg}(P) \leq n\}$. The group $\text{GL}_2(\mathbb{Z})$ acts on $V_n$ from right by $P(X)\lvert_{-n,\gamma} := (cX+d)^{n}P\Big(\frac{aX+b} {cX+d}\Big)$.
We omit the index of weight from the subscript if it is clear from the context. More generally, let $\mathbb{K}$ be a field of characteristic zero and set $V_{n, \mathbb{K}} = \mathbb{K} \otimes_{\mathbb{Z}} V_n$. 
Then the formula above defines a right $\mathbb{K}[\text{GL}_2(\mathbb{Z})]$ action on $V_{n,\mathbb{K}}$. The restriction of this action to $\Gamma$ is our primary interest. One also needs to consider the $\text{Hol}(\mathbb{H})$-module $\text{Hol}(\mathbb{H})[X]_n := \text{Hol}(\mathbb{H}) \otimes_{\mathbb{Z}} V_n$. In this article, we mostly deal with group cohomology with coefficients in a right module. As usual, $Z^{1}(\cdot, \cdot)$ denotes the collection of $1$-cocycles. If $\mathcal{X}, \mathcal{Y}$ are two sets then $\text{Fun}(\mathcal{X}, \mathcal{Y})$ is the set of all set-theoretic functions from $\mathcal{X}$ to $\mathcal{Y}$. Suppose that $\mathcal{Y}$ is a $\mathbb{K}$-vector space for some field $\mathbb{K}$. Then $\text{Fun}(\mathcal{X}, \mathcal{Y})$ has a natural $\mathbb{K}$-vector space structure arising from pointwise addition and scalar multiplication. Now suppose $G$ is an abstract group and $\mathcal{X}$, $\mathcal{Y}$ are right $G$-sets. The set $\text{Fun}(\mathcal{X}, \mathcal{Y})$ admits a right $G$-action given by 
\begin{equation}
\label{1.3}
(F.g)(x) = F(xg^{-1})g. 
\end{equation}

\section{Regularization of the Eichler-Shimura cocyle} 
\label{section2}   
\subsection{Basic constructions}
\label{section2.1}
In the theory of cusp forms, it is conventional to compute the period cocycle \eqref{1.1} with respect to the base-point at infinity so that the coefficients of the period polynomial encode the special values of $L$-functions. A naive approach extending this procedure to the Eisenstein series encounters issues regarding convergence. To circumvent the problem of dealing with integrals with the base point on the boundary, we develop our formalism in terms of a differential variant of the classical Eichler-Shimura integral that resembles the notion of Eichler integrals \cite{knopp} in the theory of Eichler cohomology.

\begin{definition}
\label{definition2.1}
Let $k$ be an integer $\geq 2$ and $f \in \text{Hol}(\mathbb{H})$. A weight $k$ Eichler-Shimura integral of $f$ is a polynomial $\mathbb{I}_{k}(\tau, f,X) \in \text{Hol}(\mathbb{H})[X]_{k-2}$ that satisfies \[\partial_{\tau}\big(\mathbb{I}_{k}(\tau, f,X)\big) = f(\tau)(X-\tau)^{k-2}d\tau\] where $\partial_{\tau}(\sum_{j=0}^{k-2} c_j(\tau)X^{j}) := \sum_{j=0}^{k-2} c'_{j}(\tau) X^{j}d\tau$. 
\end{definition}

The familiar Eichler-Shimura integral $\int^{\tau}_{\tau_0}f(\xi)(X-\xi)^{k-2}d\xi$ with $\tau_{0} \in \mathbb{H}$, provides a canonical example of an Eichler-Shimura integral in the sense described above. Two choices of Eichler-Shimura integrals differ by an element of $V_{k-2, \mathbb{C}}$. Now suppose $\mathcal{X}$ is a collection of holomorphic functions on $\mathbb{H}$. A choice of weight $k$ Eichler-Shimura integrals for $\mathcal{X}$ is a set-theoretic function $\mathbb{I}_{k}: \mathcal{X} \to \text{Hol}(\mathbb{H})[X]_{k-2}$ satisfying $\partial_{\tau}\big(\mathbb{I}_{k}(\tau, f,X)\big) = f(\tau)(X-\tau)^{k-2}d\tau$.

\begin{example}
\label{example2.2} 
Let $\tau_{0} \in \mathbb{H}$. The \textit{canonical choice} for $\mathcal{X}$ with respect to $\tau_{0}$, denoted $\mathbb{I}_{k}^{\tau_{0}}$, is given by $f \to \int_{\tau_{0}}^{\tau}f(\xi)(X-\xi)^{k-2}d\xi$.  
\end{example}

Given two choice functions for $\mathcal{X}$, say $\mathbb{I}^{(1)}_{k}$ and $\mathbb{I}^{(2)}_{k}$, we obtain a set map 
\begin{equation*}
	\mathbb{I}^{(1)}_{k} - \mathbb{I}^{(2)}_{k}: \mathcal{X} \to V_{k-2, \mathbb{C}},\hspace{.3cm} f \to \mathbb{I}^{(1)}_{k}(\tau, f, X) - \mathbb{I}^{(2)}_{k}(\tau, f, X). 
\end{equation*} 

\begin{remark}
\label{remark2.3} Let the notation be as in Definition~\ref{definition2.1}. We say $F \in \text{Hol}(\mathbb{H})$ is a weight $k$ \textit{Eichler integral} of $f$ if $F$ satisfies the differential equation 
\[\frac{d^{k-1}F}{d\tau^{k-1}} = (k-2)! F.\]
Let $\tau_{0} \in \mathbb{H}$ be fixed. Observe that $\mathbb{I}_{k}^{\tau_{0}}(\tau, f, \tau)$ is a holomorhic solution of the above differential equation. Therefore, if $\mathbb{I}_{k}(\tau, f, X)$ is an Eichler-Shimura integral, then  $\mathbb{I}_{k}(\tau, f, \tau)$ is an Eichler integral since $\frac{d^{k-1}}{d\tau^{k-1}}$ annihilates polynomials of degree $\leq k-2$.  Conversely, an Eichler integral is in the form $\mathbb{I}_{k}^{\tau_{0}}(\tau, f, \tau) + P(\tau)$ for some $P[X] \in V_{k-2, \mathbb{C}}$. Thus, it can be obtained as $\mathbb{I}_{k}(\tau, f, \tau)$ for a suitable Eichler-Shimura integral $\mathbb{I}_{k}(\tau, f, X)$. In particular, all our constructions and results can be interpreted in terms of the theory of Eichler cohomology and vice versa.    
\end{remark}

Let $\mathcal{G}$ be a congruence subgroup of $\Gamma$. Suppose that $\mathcal{X}$ is a subset of $\text{Hol}(\mathbb{H})$ which is stable under the $\lvert_k$-action of $\mathcal{G}$. The $\mathcal{G}$-actions on $\mathcal{X}$ and $V_{k-2, \mathbb{C}}$ induce a right $\mathbb{C}[\mathcal{G}]$ module structure on $\text{Fun}(\mathcal{X}, V_{k-2, \mathbb{C}})$ according to the recipe in \eqref{1.3}. Let $\mathbb{I}_{k}$ be a choice of weight $k$ Eichler-Shimura integrals for $\mathcal{X}$. Define the period function attached to the choice function $\mathbb{I}_k$ by  
\begin{equation}
\begin{aligned}
\label{2.1}
\mathsf{p}_k &: \mathcal{G} \to \text{Fun}(\mathcal{X}, V_{k-2,\mathbb{C}}),\\
\mathsf{p}_{k}(\gamma)(f) & :=  (cX+d)^{k-2}\mathbb{I}_{k}(\gamma\tau, f\lvert_{ \gamma^{-1}}, \frac {aX+b} {cX+d}) - \mathbb{I}_{k}(\tau, f,X).
\end{aligned}
\end{equation}
Let $\mathsf{p}_k^{\tau_{0}}$ denote the period function attached to the canonical choice $\mathbb{I}_{k}^{\tau_{0}}$. Then a direct calculation using the change of variables formula shows that \[\mathsf{p}_k^{\tau_{0}}(\gamma)(f) = \int_{\gamma^{-1}\tau_{0}}^{\tau_{0}} f(\xi)(X-\xi)^{k-2}d\xi\]  
consistent with the classical definition \eqref{1.1}.

\begin{lemma}
\label{lemma2.4}
Let $\mathbb{I}_k$ be a choice of Eichler-Shimura integrals for $\mathcal{X}$. Then the formula \eqref{2.1} gives rise to a well-defined cocycle on $\mathcal{G}$. Moreover, the image of $\mathsf{p}_k$ in $H^{1}\big(\mathcal{G}, \emph{Fun}(\mathcal{X},V_{k-2,\mathbb{C}})\big)$ does not depend on the choice function. 
\end{lemma}

\begin{proof}
Let $\tau_{0} \in \mathbb{H}$ be fixed and $\gamma \in \mathcal{G}$. A direct calculation using \eqref{2.1} yields  
\begin{equation*}
\mathsf{p}_{k}(\gamma) - \mathsf{p}^{\tau_{0}}_k(\gamma) = (\mathbb{I}_{k} - \mathbb{I}^{\tau_{0}}_{k})\bigl\lvert_{\gamma} - (\mathbb{I}_{k} - \mathbb{I}^{\tau_{0}}_{k})
\end{equation*}
where we are considering $\mathbb{I}_{k} - \mathbb{I}^{\tau_{0}}_{k}$ as an element of $\text{Fun}(\mathcal{X}, V_{k-2,\mathbb{C}})$. The explicit expression for $\mathsf{p}_k^{\tau_{0}}$ immediately verifies the first part of the assertion for $\mathbb{I}_k^{\tau_{0}}$. One obtains the statement for $\mathbb{I}_k$ using the identity above. It also shows that the two cocycles differ by a coboundary as desired. 
\end{proof}

Note that our setting of period function allows a nontrivial action of $\mathcal{G}$ on the set of test functions $\mathcal{X}$, contrary to the classical approach, which works with $\mathcal{G}$-invariant functions. In the sequel, we refer to the period function \eqref{2.1} as the \textit{equivariant} period function to distinguish our formalism from the classical construction. The equivariant period function is particularly appropriate for the induced perspective towards Eichler-Shimura theory (Section~\ref{section2.3}) since here one can always evaluate the period function at $\gamma = S$ even if $f$ is not $S$-invariant. Our equivariant approach also works well for the modular forms with a nebentypus character, though we do not elaborate on this aspect in the present article. Given a congruence subgroup $\mathcal{G}$ and $\mathcal{G}$-stable subset $\mathcal{X}$, Lemma~\ref{lemma2.4} yields a canonical cohomology class
\[[\mathsf{p}_k] \in H^{1}\big(\mathcal{G}, \text{Fun}(\mathcal{X},V_{k-2,\mathbb{C}})\big)\]
using the Eichler-Shimura integrals. This cohomology class provides an equivariant analog of the usual Eichler-Shimura map as explained below.

From now on, we assume that $\mathcal{X}$ consists of $\mathcal{G}$ invariant functions. Let $V$ be a right $\mathbb{K}[\mathcal{G}]$-module for some field $\mathbb{K}$. Then there is a $\mathbb{K}$-linear ismorphism \[\Phi_{\text{cy}}: Z^{1}\big(\mathcal{G}, \text{Fun}(\mathcal{X}, V)\big) \xrightarrow{\cong} \text{Fun}\big(\mathcal{X}, Z^{1}(\mathcal{G}, V)\big),\, \Phi_{\text{cy}}(P)(f)(\gamma) := P(\gamma)(f)\]
that descends to an isomorphism: 
\[\overline{\Phi}_{\text{cy}}: H^{1}\big(\mathcal{G}, \text{Fun}(\mathcal{X}, V)\big) \xrightarrow{\cong} \text{Fun}\big(\mathcal{X}, H^{1}(\mathcal{G}, V)\big).\]  
Suppose that $\mathbb{I}_{k}$ is a choice of Eichler-Shimura integrals for $\mathcal{X}$. We define the Eichler-Shimura cocyle map attached to $\mathbb{I}_k$ by 
\[\text{ES}_k: \mathcal{X} \to Z^{1}(\mathcal{G}, V_{k-2,\mathbb{C}}), \hspace{.3cm} \text{ES}_k := \Phi_{\text{cy}}(\mathsf{p}_k).\]
Moreover, the \textit{Eichler-Shimura morphism} attached to the pair $(\mathcal{G}, \mathcal{X})$ is 
\begin{equation}
\label{2.2}
[\text{ES}_k] : \mathcal{X} \to H^{1}\big(\mathcal{G},V_{k-2,\mathbb{C}}\big), \hspace{.3cm} [\text{ES}_k](f) := \overline{\Phi}_{\text{cy}}([\mathsf{p}_{k}]).
\end{equation}
Lemma~\ref{lemma2.4} shows, the map \eqref{2.2} does not depend on the choice of $\mathbb{I}_k$. Observe that the cocyle $\Phi_{\text{cy}}(\mathsf{p}^{\tau_{0}}_k)(f)$ equals $\text{ES}_k^{\tau_{0}}(f)$ defined in \eqref{1.1}. Therefore, the Eichler-Shimura morphism constructed above coincides with the classical Eichler-Shimura map given by \eqref{1.1}. Here, our construction \eqref{2.2} does not assume a linear structure on $\mathcal{X}$. Now suppose $\mathbb{K}$ is a subfield of $\mathbb{C}$ and $\mathcal{X}$ is a $\mathbb{K}$-subspace of $\text{Hol}(\mathbb{H})$. One can calculate the Eichler-Shimura morphism using a canonical choice function to discover that $[\text{ES}_k]$ is automatically $\mathbb{K}$-linear. Similarly, if $\mathcal{X}$ is closed under a double coset operator $[\mathcal{G}\alpha\mathcal{G}]$ for some $\alpha \in \text{GL}_2^{+}(\mathbb{Q})$, then a straight-forward computation using the canonical choice verifies that $[\text{ES}_k]$ commutes with the action of $[\mathcal{G}\alpha\mathcal{G}]$ on its domain and codomain \cite[8.3]{shimura94}. Next, we discuss the equivariance of the Eichler-Shimura map under the action of a larger group, which plays a crucial role in our calculations.

Let $\mathcal{H}$ be an abstract group and $V$ be a right $\mathbb{K}[\mathcal{H}]$-module where $\mathbb{K}$ is a field of characteristic $0$. Suppose that $\mathcal{G}$ is a subgroup of $\mathcal{H}$. Given $h \in \mathcal{H}$, there is an isomorphism at the level of cocycles
\[Z^{1}(\mathcal{G}, V) \to Z^{1}(h^{-1}\mathcal{G}h, V); \hspace{.3cm} \text{$P \mapsto P|_{h}$, $P\lvert_{h}(x) := P(hxh^{-1})h$}\]
which gives rise to an isomorphism $H^{1}(\mathcal{G}, V) \cong H^{1}(h^{-1}\mathcal{G}h, V)$. Suppose that $\mathcal{G}$ is a normal subgroup of $\mathcal{H}$. Then $[P] \mapsto [P\lvert_{h}]$ is a well-defined right $\mathcal{H}$ action on $H^{1}(\mathcal{G},V)$. In this situation, the action of $\mathcal{H}$ on $H^{1}(\mathcal{G}, V)$ factors through $\mathcal{G}$, i.e., the elements of $\mathcal{G}$ act trivially on cohomology. If $[\mathcal{H}: \mathcal{G}] < \infty$, then the canonical restriction homomorphism induces an isomorphism $H^{1}(\mathcal{H}, V) \cong H^{1}(\mathcal{G}, V)^{\mathcal{H}}$; see \cite[III.10]{ksbrown}. 

\begin{lemma}
\label{lemma2.5}
Let the notation be as in \eqref{2.2}. Suppose that $\mathcal{H}$ is a congruence subgroup containing $\mathcal{G}$ as a normal subgroup. Further, assume that $\mathcal{X}$ is stable under the action of $\mathcal{H}$. Then $[\emph{ES}_k]$ is compatible with the $\mathcal{H}$-action on its domain and codomain. 
\end{lemma} 

\begin{proof}
Follows from a direct calculation using a canonical choice of Eichler-Shimura integrals for $\mathcal{X}$. 
\end{proof}

With the notation in the lemma, let $\mathcal{Y}$ be a subset of $\mathcal{X}^{\mathcal{H}}$. Then the Eichler-Shimura morphisms for $(\mathcal{G}, \mathcal{X})$ and $(\mathcal{H}, \mathcal{Y})$ fit into a commutative diagram 
\begin{equation}
\label{2.3}
\begin{tikzcd}
\mathcal{Y} \arrow[r, "{[\text{ES}_k]}"]\arrow[d] & H^{1}(\mathcal{H}, V_{k-2, \mathbb{C}}) \arrow[d]\\
\mathcal{X} \arrow[r, "{[\text{ES}_k]}"] & H^{1}(\mathcal{G}, V_{k-2, \mathbb{C}})
\end{tikzcd}
\end{equation}
where the vertical arrows are restrictions. In this article, we use the diagram above to reduce the problem of computing the Eichler-Shimura map for a general subgroup to a principal level contained in it.

\subsection{Base point at the boundary}
\label{section2.2}
Our formalism of the Eichler-Shimura morphism does not assume any growth condition at the boundary for the test functions in $\mathcal{X}$. However, to construct Eichler-Shimura integrals with the base point at the boundary, one needs to impose a suitable growth condition at the boundary.  The closure of $\mathbb{H}$ in the Riemann sphere equals $\overline{\mathbb{H}} := \mathbb{H} \cup \mathbb{R} \cup \{\infty\}$. We endow the extended upper half-plane $\mathbb{H}^{\ast} \subseteq \overline{\mathbb{H}}$ with the subspace topology arising from $\widehat{\mathbb{C}}$ because it is favorable for discussing path integrals. To handle the convergence problem at the boundary, however, one needs a special family of punctured neighborhoods around boundary points that resemble the Satake topology for the Baily-Borel compactification \cite[2.1]{goresky} and transform well under the action of $\Gamma$. Readers may note that our choice of compactification differs from the standard automorphic method \cite{harder} that prefers the Borel-Serre compactification over the Baily-Borel approach. 

An \textit{admissible triangle} based at $x \in \mathbb{P}^{1}(\mathbb{Q})$ refers to a geodesic triangle on $\mathbb{H}^{\ast}$ whose one vertex is at $x$ and two other vertices lie in $\mathbb{H}$. 
  
\begin{definition}
\label{definition2.6}
A triangular neighborhood of $x \in \mathbb{P}^{1}(\mathbb{Q})$ is the intersection of an open neighborhood of $x$ in $\mathbb{H}^{\ast}$ with the interior of an admissible triangle based at $x$.
\end{definition}

Observe that the action of $\gamma \in \Gamma$ maps a triangular neighborhood of $x$ onto a triangular neighborhood of $\gamma x$.  

\begin{example}
\label{example2.7} 
This example describes the triangular neighborhoods of $\infty$. Let $t_1, t_2, h \in \mathbb{R}$ so that $t_1 < t_2$ and $h > 0$. Set \[\mathcal{N}^{h}_{t_1,t_2} := \{\tau \in \mathbb{H} \mid \substack{t_1 < u(\tau) < t_2,\, v(\tau) > h} \}.\]
If $h > \frac{t_2 - t_1}{2}$ then $\mathcal{N}^{h}_{t_1,t_2}$ is a triangular neighborhood of $\infty$. In general, if $W$ is a nonempty triangular neighborhood of $\infty$ then there exists $t_1 < t_2$ and $h_0 > 0$ so that $W = \mathcal{N}_{t_1, t_2}^{h} \cup \mathcal{C}_h$ for each $h \geq h_0$ where $\mathcal{C}_h$ is a relatively compact subset of $\mathbb{H}$, i.e., the closure of $\mathcal{C}_h$ in $\mathbb{H}$ is compact.   
\end{example}

We first introduce a scale for the growth of a function in the triangular neighborhoods of $\infty$. Let $\alpha$ be a fixed real number. A function $f \in \text{Hol}(\mathbb{H})$ \textit{decays with exponent $\alpha$} at $\infty$ if $\lvert f(\tau)\rvert\lvert\tau\rvert^{\alpha}$ is bounded on each triangular neighborhood of $\infty$. Note that the definition does not require the bound to be uniform.
Write \[\mathfrak{F}_{\alpha}:= \{f \in \text{Hol}(\mathbb{H}) \mid \text{$f$ decays with exponent $\alpha$ at $\infty$}\}.\] It is clear that $\mathfrak{F}_{\alpha}$ is a $\mathbb{C}$-subspace of $\text{Hol}(\mathbb{H})$. Moreover, $\mathfrak{F}_{\alpha_2} \subseteq \mathfrak{F}_{\alpha_1}$ if $\alpha_1 \leq \alpha_2$. Set  $\mathfrak{F}_{\infty}:= \bigcap_{\alpha \in \mathbb{R}} \mathfrak{F}_{\alpha}$. 

The boundary points in $\mathbb{P}^{1}(\mathbb{Q}) -\{\infty\}$ naturally play a role in the theory due to the action of $\Gamma$ on the extended upper half plane. Let $\alpha$ be a real number and $x_0 \in \mathbb{P}^{1}(\mathbb{Q}) -\{\infty\}$. A function $f \in \text{Hol}(\mathbb{H})$ \textit{grows with exponent $\alpha$} at $x_0$ if $\lvert f(\tau)\rvert\lvert\tau - x_0\rvert^{\alpha}$ is bounded on each triangular neighborhood of $x_0$. Define \[\mathfrak{F}_{\alpha}^{x_0} := \{f \in \text{Hol}(\mathbb{H}) \mid \text{$f$ grows with exponent $\alpha$ at $x_0$}\}.\] It is clear that $\mathfrak{F}^{x_0}_{\alpha}$ is a $\mathbb{C}$-subspace of $\text{Hol}(\mathbb{H})$. Note that $\mathfrak{F}_{\alpha_1}^{x_0} \subseteq \mathfrak{F}_{\alpha_2}^{x_0}$ if $\alpha_1 \leq \alpha_2$. Put 
$\mathfrak{F}_{-\infty}^{x_0} := \bigcap_{\alpha \in \mathbb{R}} \mathfrak{F}_{\alpha}^{x_0}$. 

The Fourier expansion of a weakly holomorphic modular form at $\infty$ involves a principal part of the Laurent expansion along with a constant term (Example~\ref{example2.8}). For $n \geq 1$, let $\mathcal{J}_n := \{(\kappa_1, \cdots, \kappa_n) \in \mathbb{R}_{\geq 0}^{n} \mid \kappa_1 < \cdots < \kappa_n\}$. For $\kappa = (\kappa_1, \cdots, \kappa_n) \in \mathcal{J}_n$ define $\mathsf{EP}(\kappa)$ to be the $\mathbb{C}$-subspace of $\text{Hol}(\mathbb{H})$ spanned by $\big\{\textbf{e}(-\kappa_r \tau) \mid 1 \leq r \leq n\big\}$. The subspace $\mathsf{EP}(0)$ consists of the constant functions. Define the collection of exponential polynomials by \[\mathsf{EP} := \bigcup_{\substack{\kappa \in \mathcal{J}_n,\\n \geq 1}} \mathsf{EP}(\kappa).\] Since $\mathbb{R}_{\geq 0}$ is totally ordered $\mathsf{EP}$ is a $\mathbb{C}$-subspace of $\text{Hol}(\mathbb{H})$. Note that the polynomials in $\mathsf{EP}$ have unique holomorphic extensions to $\mathbb{C}$. For $\alpha \in (0, \infty]$, define 
\[\mathsf{EP}\mathfrak{F}_{\alpha}:= \mathsf{EP} + \mathfrak{F}_{\alpha} \subseteq \text{Hol}(\mathbb{H}).\] One easily checks that $\mathsf{EP} \cap \mathfrak{F}_{\alpha} = \{0\}$, that is to say, $\mathsf{EP}\mathfrak{F}_{\alpha} = \mathsf{EP} \oplus \mathfrak{F}_{\alpha}$. We write $f_{\infty}$, resp. $f_c$, for the $\mathsf{EP}$-component, resp. $\mathfrak{F}_{\alpha}$ component, of a function $f \in \mathsf{EP}\mathfrak{F}_{\alpha}$.    

\begin{example}
\label{example2.8}
Let $f$ be a weakly holomorphic modular form on a congruence subgroup $\mathcal{G}$ containing $\Gamma(N)$. Then $f$ admits a $q$-series expansion 
$f(\tau) = \sum_{j = -\infty}^{\infty} a(j)\textbf{e}(j\tau/N)$ 
where $\big(a(j)\big)_{j \in \mathbb{Z}}$ is a sequence of complex numbers so that $a(j) = 0$ for $j <<0$ and $\sum_{j \geq 1} a(j)z^j$ has radius of convergence $1$ on the open unit disc. If $f$ is a holomorphic modular form, then $a(n) = 0$ for each $n < 0$. An easy calculation using the Fourier expansion above verifies that $f \in \mathsf{EP}\mathfrak{F}_{\infty}$. Moreover, $f \in \mathbb{C}\mathfrak{F}_{\infty}:= \mathbb{C} \oplus \mathfrak{F}_{\infty}$ whenever $f$ is holomorphic.  
\end{example}

Next, we note down basic facts about the action of conformal automorphisms on growth families. 

\begin{lemma}
\label{lemma2.9}
Let $k \in \mathbb{Z}$, $\alpha \in (0, \infty]$, and $\gamma \in \Gamma$. 
\begin{enumerate}[label=(\roman*), align=left, leftmargin=0pt]
\item If $\gamma \in \Gamma_{\infty}$ then both $\mathsf{EP}$ and $\mathfrak{F}_{\alpha}$ are stable under $\lvert_{k,\gamma}$. Moreover, we have $(f\lvert_{k,\gamma})_{\infty} = f_{\infty}\lvert_{k,\gamma}$ and  $(f\lvert_{k,\gamma})_c = f_c\lvert_{k,\gamma}$, for each $\gamma \in \Gamma_{\infty}$ and $f \in \mathsf{EP}\mathfrak{F}_{\alpha}$.  
\item Now suppose $\gamma \in \Gamma - \Gamma_{\infty}$ and $f \in \mathfrak{F}_{\alpha}$. Then $f\lvert_{k,\gamma} \in \mathfrak{F}^{-\frac{d}{c}}_{k - \alpha}$.  
\end{enumerate}  
\end{lemma}
\begin{proof}
Follows from the definition of the growth families. 
\end{proof}

Before beginning the discussion about Eichler-Shimura integrals with base point at $\infty$, one needs to fix a convention regarding the admissible family of paths for path integrals having endpoint on the boundary. Let $\tau_1, \tau_2 \in \mathbb{H}^{\ast} - \{\infty\}$. A \textit{good} path between $\tau_1$ and $\tau_2$ refers to a piecewise smooth path on $\mathbb{H}^{\ast}$ that maps $(0,1)$ inside $\mathbb{H}$. The unique good path joining $\tau \in \mathbb{H}^{\ast}$ and $\infty$ is the geodesic path between $\tau$ and $\infty$. In the sequel, the path integrals are along good paths joining the endpoints.

\begin{lemma}
\label{lemma2.10}
\begin{enumerate}[label=(\roman*), align=left, leftmargin=0pt]
\item Let $\tau \in \mathbb{H}$ and $x \in \mathbb{P}^{1}(\mathbb{Q}) - \{\infty\}$. Suppose that $f \in \mathfrak{F}^{x}_{0}$, i.e., $f$ is bounded on each triangular neighborhood of $x$. Then $\int^{\tau}_{x}f(\xi)d\xi$ does not depend on the choice of a good path joining the endpoints.  

\item Now suppose $x_1, x_2 \in \mathbb{P}^{1}(\mathbb{Q}) - \{\infty\}$. Suppose that $f \in \mathfrak{F}^{x_1}_{0} \cap \mathfrak{F}^{x_2}_{0}$. Then $\int^{x_{2}}_{x_{1}}f(\xi)d\xi$ does not depend on the choice of good path.  
\end{enumerate}
\end{lemma}

\begin{proof}
The first part is an easy exercise using piece-wise smoothness of paths. The second part follows immediately from the first part. 
\end{proof}

Let $k \geq 2$ and $\alpha \in [k, \infty]$. Suppose $f \in \mathsf{EP}\mathfrak{F}_{\alpha}$. Write $f = f_{\infty} + f_{c}$. Note that $f_c(\tau)\tau^{r} \in \mathfrak{F}_{2}$ for each $0 \leq r \leq k-2$. Put 
\begin{gather*}
\mathbb{I}^{\infty}_{k}(\tau,f,X) := \int^{\tau}_{0}f_{\infty}(\xi)(X-\xi)^{k-2}d\xi,\\ \mathbb{I}^{c}_{k}(\tau,f,X) := -\int^{\infty}_{\tau}f_{c}(\xi)(X-\xi)^{k-2}d\xi.
\end{gather*}
The integrals above give rise to well-defined elements in $\text{Hol}(\mathbb{H})[X]_{k-2}$ that satisfy  $\partial_{\tau}\big(\mathbb{I}_{k}^{\infty}(\tau,f,X)\big) = f_{\infty}(\tau)(X-\tau)^{k-2}d\tau$ and $\partial_{\tau}\big(\mathbb{I}^{c}_{k}(\tau,f,X)\big) = f_{c}(\tau)(X-\tau)^{k-2}d\tau$. Therefore 
\begin{equation}
\label{2.4}
\mathbb{I}^{*}_{k}(\tau, f, X) = \mathbb{I}^{\infty}_{k}(\tau,f,X) + \mathbb{I}^{c}_{k}(\tau,f,X)
\end{equation} 
is an Eichler-Shimura integral of $f$ in the sense of Definition~\ref{definition2.1}. We refer to $\mathbb{I}^{*}_{k}(\tau, f, X)$ as the Eichler-Shimura integral of $f$ with base point at $\infty$. This strategy to regularize the Eichler-Shimura integral resembles Brown's construction of period polynomials \cite[Part I]{brown} where he interprets $\mathbb{I}^{\infty}_{k}(\tau,f, X)$ as an integration on the tangent plane.

Let $\mathcal{G}$ be a congruence subgroup of $\Gamma$ and $\mathcal{X} \subseteq \mathsf{EP}\mathfrak{F}_{\alpha}$ be a collection of functions stable under the $\lvert_{k}$-action of $\mathcal{G}$. Suppose that $\mathsf{p}^{*}_{k}$ is the period function \eqref{2.1} determined by the choice of Eichler-Shimura integrals $\mathbb{I}^{*}_{k}$. Next, we provide an explicit formula for the period function. For convenience write \[f_{\infty}(\gamma):= (f\lvert_{\gamma})_{\infty},\hspace{.2cm} f_{c}(\gamma):= (f\lvert_{\gamma})_c.\hspace{.3cm}(\substack{f \in \mathcal{X}, \gamma \in \mathcal{G}})\]

\begin{proposition}
\label{proposition2.11}
Let $f \in \mathcal{X}$ and $\gamma \in \mathcal{G}$. Then  
\begin{equation*}
\begin{aligned}
\mathsf{p}_{k}^{*}(\gamma)(f) & = \int_{\gamma^{-1}0}^{\tau_0} f_{\infty}(\gamma^{-1})\lvert_{\gamma}(\xi)(X-\xi)^{k-2}d\xi - \int_{0}^{\tau_0} f_{\infty}(\xi)(X-\xi)^{k-2}d\xi \\
& - \int_{\tau_0}^{\gamma^{-1}\infty} f_c(\gamma^{-1})\lvert_{\gamma}(\xi)(X-\xi)^{k-2}d\xi + \int_{\tau_0}^{\infty} f_c(\xi)(X-\xi)^{k-2}d\xi  
\end{aligned} 
\end{equation*} 
for each $\tau_{0} \in \mathbb{H}$.  
\end{proposition}
\begin{proof}
Let $\tau_0 \in \mathbb{H}$. One uses \eqref{2.1} with $\tau = \tau_{0}$ to conclude that   
\begin{equation*}
\begin{aligned}
\mathsf{p}_{k}^{*}(\gamma)(f) & = (cX+d)^{k-2} \int^{\gamma\tau_0}_{0} f_{\infty}(\gamma^{-1})(\xi)(\frac{aX+b}{cX+d} - \xi)^{k-2} d\xi - \mathbb{I}^{\infty}_k(\tau_0,f,X)\\
& - (cX+d)^{k-2}\int_{\gamma\tau_0}^{\infty} f_c(\gamma^{-1})(\xi)(\frac{aX+b}{cX+d}-\xi)^{k-2}d\xi - \mathbb{I}^{c}_k(\tau_0,f,X). 
\end{aligned}
\end{equation*}
To prove the proposition, one must show that the change of variable $\xi \to \gamma\xi$ transforms the first, resp. third, term in the RHS above into the first, resp. third, term in RHS of the identity in the assertion. We take the path of integration in the first integral of the RHS to be the geodesic arc joining $0$ and $\gamma\tau_0$ so that the transformed path is good if $\gamma^{-1}0 = \infty$. Now, a routine computation using Lemma~\ref{lemma2.9} verifies that the change of variable goes through without any problem.   
\end{proof}  

The theory developed in this section also offers a cohomological vista of the recent developments on weakly holomorphic modular forms \cite{mockmodular}. For instance, one can use the proposition above to compute the parabolic cocycle attached to weakly holomorphic modular forms.
  
\begin{corollary}
\label{corollary2.12}
Let $\gamma \in \mathcal{G}_{\infty}$ and $f \in \mathcal{X}$. We have \[\mathsf{p}_{k}^{*}(\gamma)(f) = \int_{\gamma^{-1}0}^{0}f_{\infty}(\xi)(X-\xi)^{k-2}d\xi.\] In particular, if $f_{\infty}$ is constant, then $\mathsf{p}_{k}^{*}(\gamma)(f)  = \frac{f_{\infty}}{k-1}\big((X+r)^{k-1}-X^{k-1}\big)$ where $\gamma^{-1}0 = -r$. 
\end{corollary}
 
\begin{proof}
Follows from the formula in Proposition~\ref{proposition2.11}. 
\end{proof}

\subsection{The induced picture}
\label{section2.3}
We begin with a convenient description of Shapiro's lemma \cite[p.62]{neukirch et al}. Let $\mathcal{G}_1$ be an abstract group and $\mathcal{G}_2$ be a finite index subgroup of $\mathcal{G}_1$. Suppose that $V$ is a right $\mathbb{K}[\mathcal{G}_2]$-module where $\mathbb{K}$ is a field of characteristics zero. We consider $\mathbb{K}[\mathcal{G}_1]$ as a right $\mathbb{K}[\mathcal{G}_2]$ module and set  $\text{Ind}^{\mathcal{G}_1}_{\mathcal{G}_2}(V) = \text{Hom}_{\mathbb{K}[\mathcal{G}_2]}(\mathbb{K}[\mathcal{G}_1], V)$. The induced module carries a right $\mathcal{G}_1$-action defined by $(f\lvert_{g_1})(x) = f(g_1x)$.  Shapiro's lemma states that the canonical homomorphism of $\mathcal{G}_2$-modules $\text{Ind}^{\mathcal{G}_1}_{\mathcal{G}_2}(V) \to V$ described by $f \mapsto f(\text{Id})$  
gives rise to a $\mathbb{K}$-linear isomorphism \[H^{\ast}\big(\mathcal{G}_1, \text{Ind}^{\mathcal{G}_1}_{\mathcal{G}_2}(V)\big) \cong H^{\ast}(\mathcal{G}_2, V).\] If the $\mathbb{K}[\mathcal{G}_2]$-module structure on $V$ extends to a $\mathbb{K}[\mathcal{G}_1]$-structure, then \[\text{Ind}^{\mathcal{G}_1}_{\mathcal{G}_2}(V) \cong \text{Fun}(\mathcal{G}_2\backslash \mathcal{G}_1, V)\] as $\mathbb{K}[\mathcal{G}_1]$-module via the assignment $f \to F$ where $F(\mathcal{G}_2g_1) := f(g_1^{-1})g_1$.

Let $\mathcal{G}$ be a finite index subgroup of $\Gamma$. Suppose that $V$ is a right $\mathbb{K}[\Gamma]$-module. We consider the isomorphism $\text{sh}^{1}: H^{1}\big(\Gamma, \text{Fun}(\mathcal{G}\backslash \Gamma, V)\big) \xrightarrow{\cong} H^{1}(\mathcal{G}, V)$ described by 
\begin{equation}
\label{2.5}
[P] \to [\psi(P)]; \hspace{.3cm} \psi(P)(\gamma) := P(\gamma)(\mathcal{G}\text{Id}).\hspace{.3cm}(\substack{\gamma \in \mathcal{G}})
\end{equation}
In practice, working with the cocycles on $\Gamma$ is more convenient since each cocycle is determined by its value on $\{T, S\}$. One pulls back the action of the double coset operators on $H^{1}(\mathcal{G}, V)$ to obtain an action of the Hecke algebra attached to $\mathcal{G}$ on $H^{1}(\Gamma, \text{Fun}(\mathcal{G}\backslash \Gamma, V))$; cf. \cite[5]{pasol-popa}. Our next job is to describe the Eichler-Shimura morphism in the induced picture. There are explicit constructions of the inverse of the Shapiro isomorphism at the level of cocycles \cite[3]{paskunas-quast}. We, however, directly construct the induced Eichler-Shimura cocycle using the equivariant period function. Let $\mathcal{G}$ be a congruence subgroup of $\Gamma$ and $\mathcal{X}$ be a collection of $\mathcal{G}$-invariant holomorphic functions. Write $\mathcal{X}_{\Gamma} = \{f\lvert_{\gamma} \mid f \in \mathcal{X}, \gamma \in \Gamma\}$. In other words, $\mathcal{X}_{\Gamma}$ is the smallest $\Gamma$-stable subset of $\text{Hol}(\mathbb{H})$ containing $\mathcal{X}$. Let $\mathbb{I}_{k}$ be a choice of Eichler-Shimura integrals for the $\Gamma$-stable subset $\mathcal{X}_{\Gamma}$ and $\mathsf{p}_k$ be the equivariant period function of $\mathbb{I}_k$. Then, the induced Eichler-Shimura cocycle map attached to $\mathbb{I}_{k}$ is defined by
\begin{equation}
\label{2.6}
\begin{aligned}
&\text{Ind-ES}_{k} : \mathcal{X} \to Z^{1}\big(\Gamma, \text{Fun}(\mathcal{G} \backslash \Gamma, V_{k-2,\mathbb{C}})\big),\\ 
& \hspace{.3cm}\text{Ind-ES}_k(f)(\gamma)(\mathcal{G}\sigma) = \mathsf{p}_k(\gamma)(f\lvert_{\sigma}). \hspace{.3cm}(\substack{\gamma \in \mathcal{G}, \sigma \in \Gamma})
\end{aligned}
\end{equation}
It is clear that $\psi \circ \text{Ind-ES}_k = \text{ES}_k$ where $\text{ES}_k$ is the cocycle map attached to $(\mathcal{G}, \mathcal{X})$ for the choice of Eichler-Shimura integrals determined by $\mathbb{I}_{k}$. Define the induced Eichler-Shimura morphism for the pair $(\mathcal{G}, \mathcal{X})$ by   
\[\text{Ind-}[\text{ES}_k]: \mathcal{X} \to H^{1}\big(\Gamma, \text{Fun}(\mathcal{G} \backslash \Gamma, V_{k-2,\mathbb{C}})\big), \hspace{.3cm} f \mapsto [\text{Ind-ES}_{k}(f)].\]
The comparison between the induced and non-induced cocycle maps above at once shows that 
$\text{sh}^{1} \circ \text{Ind-}[\text{ES}_k] = [\text{ES}_{k}]$. In particular, $\text{Ind-}[\text{ES}_k]$ does not depend on the choice of cocycle map. Suppose that $\mathcal{X}$ consists of weight $k$ weakly holomorphic modular forms on $\mathcal{G}$. Then $\mathcal{X}_{\Gamma}$ is contained in the subspace of weakly holomorphic modular forms on $\Gamma(N)$ for a principal level $\Gamma(N)$ contained in $\mathcal{G}$. It follows that $\mathcal{X}_{\Gamma} \subseteq \mathsf{EP}\mathfrak{F}_{\infty}$. Thus, one can choose $\mathbb{I}_{k}$ to be $\mathbb{I}_{k}^{\ast}$, i.e., the choice of Eichler-Shimura integrals determined by the base point at $\infty$. The induced cocycle map attached to the base point at infinity can be described using Fourier coefficients and $L$-values. We explicitly write down this procedure for holomorphic modular forms. 

Let $\mathcal{G}$ be a congruence subgroup containing $\Gamma(N)$. For $f \in \mathcal{M}_k(\mathcal{G})$ define the completed $L$-function of $f$ by 
\begin{equation}
\label{2.7}
L^{*}(f,s) = (2\pi)^{-s}\Gamma(s)L(f,s) := \int^{\infty}_{0}\big(f(it) - f_{\infty}\big)t^{s-1}dt
\end{equation} 
where $f_{\infty}$ is the constant term of the Fourier expansion as in Section~\ref{section2.2}. The integral defining $L^{*}(f,s)$ is absolutely convergent in the region $\text{Re}(s) > k$. Moreover, we have the following identity involving entire integrals: 
\begin{equation*}
	\begin{aligned}
		&L^{*}(f,s)  = \int_{1}^{\infty} \big(f(it) - f_{\infty}\big)t^{s-1}dt \hspace{2cm}(\substack{\text{Re}(s)> k})\\
		&+ \int_{0}^{1} \big(f(it) - (f\lvert_{S^{-1}})_{\infty}(it)^{-k}\big)t^{s-1}dt + i^{-k}\frac{(f\lvert_{S^{-1}})_{\infty}}{s-k} - \frac{f_{\infty}}{s}. 
	\end{aligned}	
\end{equation*}  
In particular, $L^{*}(f,s)$ extends to a meromorphic function on the whole complex plane with at most two simple poles supported at $\{0,k\}$. It also satisfies a functional equation $L^{\ast}(f,s) = i^k L^{\ast}(f\lvert_{S}, k-s)$. 

Suppose that $\mathcal{X}$ is a subset of $\mathcal{M}_{k}(\mathcal{G})$. Then $\mathcal{X}_{\Gamma} \subseteq \mathcal{M}_k\big(\Gamma(N)\big)$. In particular $\mathcal{X}_{\Gamma}$ is a subset of $\mathbb{C}\mathfrak{F}_{\infty}$. The following lemma provides an explicit formula for the induced cocycle map \eqref{2.6}, denoted $\text{Ind-}\text{ES}^{\ast}_k$, attached to $\mathbb{I}_{k}^{\ast}$.    

\begin{proposition}
\label{proposition2.13}\emph{\cite[8]{pasol-popa}}   
Let $f \in \mathcal{X}$ and $\sigma \in \Gamma$. Then 
\begin{align*}
\mathsf{p}_k^{*}(T) (f\lvert_{\sigma})  & = (f\lvert_{\sigma})_{\infty} \frac{(X+1)^{k-1} - X^{k-1}}{k-1},\\
\mathsf{p}_k^{*}(S) (f\lvert_{\sigma}) & = \sum_{r=0}^{k-2}i^{1-r}\binom{k-2}{r}L^{*}(f\lvert_{\sigma}, r+1)X^{k-2-r}. 
\end{align*}

\end{proposition} 
\begin{proof}
The first identity is a direct consequence of Corollary~\ref{corollary2.12}. We appeal to the formula given in Proposition~\ref{proposition2.11} to prove the second identity. In our context $\tau_{0} = i$ and $\gamma = S$. For brevity write $F = f\lvert_{\sigma}$. Here $F_{c}(S^{-1})\lvert_S (\tau) = F(\tau) - \tau^{-k}(F\lvert_{S^{-1}})_{\infty}$. A simple calculation using the identity below \eqref{2.7} yields 
\begin{equation*}
\begin{aligned}
& \int_{i}^{\infty} F_c(\xi)(X-\xi)^{k-2}d\xi - \int_{i}^{0} F_c(S^{-1})\lvert_S(\xi)(X-\xi)^{k-2}d\xi\\
& = \sum_{r=0}^{k-2} i^{1-r} \binom{k-2}{r} \big(L^{*}(F, r+1) + i^{-k} \frac{(F\lvert_{S^{-1}})_{\infty}}{k-r-1} + \frac{F_{\infty}}{r+1}\big)X^{k-2-r}. 
\end{aligned}
\end{equation*} 
The proof of the second formula is now clear. 
\end{proof}

Let the notation be as in the discussion above the proposition. Suppose that $V_{\infty, \mathbb{C}}$ is the field of rational functions $\mathbb{C}(X)$. We endow $V_{\infty, \mathbb{C}}$ with a right $\Gamma$-module structure by $P(X)\lvert_{\gamma} = (cX+d)^{k-2}P(\frac{aX+b}{cX+d})$
so that the inclusion $V_{k-2, \mathbb{C}} \hookrightarrow V_{\infty, \mathbb{C}}$ is a homomorphism of $\mathbb{C}[\Gamma]$-modules. Thus, one can consider $\text{Ind-}\text{ES}^{\ast}_k$ as a map from $\mathcal{X}$ to $Z^{1}\big(\Gamma, \text{Fun}(\mathcal{G}\backslash \Gamma, V_{\infty, \mathbb{C}})\big)$. Define $\text{Ind-}\widehat{\text{ES}}_{k} : \mathcal{X} \to Z^{1}\big(\Gamma, \text{Fun}(\mathcal{G} \backslash \Gamma, V_{\infty,\mathbb{C}})\big)$ as
\begin{equation*}
\text{Ind-}\widehat{\text{ES}}_k(f)(\gamma) = \text{Ind-}\text{ES}^{\ast}_k(f)(\gamma) + p_{0}(f)\lvert_{\gamma} - p_0(f)
\end{equation*}
where $p_0(f) \in \text{Fun}(\mathcal{G} \backslash \Gamma, V_{\infty,\mathbb{C}})$ is given by $p_0(f)(\mathcal{G}\sigma) := - \frac{(f\lvert_{\sigma})_{\infty}X^{k-1}}{k-1}$. Then $\text{Ind-}\widehat{\text{ES}}_{k}(f)(T)$ vanishes for each $f$ while $\text{Ind-}\widehat{\text{ES}}_{k}(f)(S)$ agrees with the definition of the period of $f$ in the work of Pasol and Popa cited above. Note that $\text{Ind-}\widehat{\text{ES}}_{k}$ and $\text{Ind-}\text{ES}^{\ast}_{k}$ induce the same map at the level of cohomology with the coefficients in $\text{Fun}(\mathcal{G} \backslash \Gamma, V_{\infty,\mathbb{C}})$. Our work provides a counterpart of their construction with coefficients in $\text{Fun}(\mathcal{G} \backslash \Gamma, V_{k-2,\mathbb{C}})$ and verifies that at the level of cohomology it induces the same map as \eqref{1.2} modulo the Shapiro isomorphism. 

\section{Eisenstein series}
\label{section3}
This section reviews the spectral Eisenstein series on a general congruence subgroup following the author's recent account of the theory \cite{myself}. We also introduce the geometric family of Eisenstein series central to our work and explain how to derive the Fourier expansion of the Maass Eisenstein series as a power series using contour integration.   

\subsection{Spectral Eisenstein series}
\label{section3.1}
Let $k$ be an integer $\geq 2$ and $\mathcal{G}$ be a congruence subgroup of $\Gamma$. Suppose that $x \in \mathbb{P}^{1}(\mathbb{Q})$ and $\sigma_x \in \Gamma$ is a scaling matrix for $x$, i.e., $x = \sigma_x \infty$. If $k$ is odd, then additionally assume that $-\text{Id} \notin \mathcal{G}$ and $x$ is a regular cusp. The spectral Eisenstein series attached to the cusp $x$ and the scaling matrix $\sigma_x$ is given by 
\begin{equation}
\label{3.1}
E_{k,x}(\tau;s) := \sum_{\gamma \in \mathcal{G}_x\backslash\mathcal{G}} \mathbf{j}(\sigma_x^{-1}\gamma, \tau)^{-k}v(\sigma_x^{-1}\gamma\tau)^{s} \hspace{.3cm}\substack{(\text{Re}(k+2s) > 2)}
\end{equation}
where $\mathbf{j}(\gamma, \tau) = c\tau + d$ for $\gamma = \begin{pmatrix}
	a & b\\
	c & d
\end{pmatrix} \in \Gamma$. If $\sigma_x$ and $\sigma_x'$ are two different scaling matrices for $x$, then the series attached to $\sigma_x$ and $\sigma_x'$ differ by a sign factor of the form $(\pm 1)^k$. In particular, if $k$ is odd, then the series above may depend on the choice of scaling matrix. The series in RHS of \eqref{3.1} admits analytic continuation to the larger domain $\text{Re}(k+2s) > 1$, allowing us to specialize the parameters at $(k,s) = (2,0)$. For simplicity we abbreviate $E_{k,x}(\tau;0)$ as $E_{k,x}(\tau)$. An explicit calculation with the Fourier series shows 
\[E_{k,x}(\tau) \in \begin{cases}
	\text{Hol}(\mathbb{H}), & \text{$k \geq 3$,}\\
	\frac{\alpha_x}{\pi v(\tau)} + \text{Hol}(\mathbb{H}), & \text{$k = 2$}
\end{cases}\]
where $\alpha_x$ is a nonzero rational number \cite[p.16]{myself}. Moreover, we have 
\begin{equation}
\label{3.2}
\pi_{\infty}(E_{k,x} \lvert_{\gamma}) := \lim_{v \to \infty} E_{k,x}(iv) = \begin{cases}
		(\pm 1)^k, & \text{if $\gamma \infty = x$ in $\mathcal{C}(\mathcal{G})$;}\\
		0, & \text{otherwise.} 
	\end{cases} \hspace{.3cm}(\substack{\gamma \in \Gamma})
\end{equation} 
Set 
\begin{equation*}
	\mathcal{B}_{\mathcal{G}, k} =
	\begin{cases}
		\{E_{k,x} \mid x \in \mathcal{C}(\mathcal{G})\}, & \substack{\text{if $k$ is even and $\geq 4$;}}\\
		\{E_{k,x} \mid x \in \mathcal{C}_{\infty}(\mathcal{G})\}, & \substack{\text{if $k$ is odd and $-\text{Id} \notin \mathcal{G}$;}}\\
		\{E_{k,x} - \frac{\alpha_x} {\alpha_{\infty}} E_{k,\infty}\mid x \in \mathcal{C}(\mathcal{G}) - \{\infty\}\}, & \substack{\text{if $k=2$;}}\\
		\emptyset, & \substack{\text{otherwise}} 
	\end{cases}
\end{equation*}
where in the second case, we write the series for a fixed choice of scaling matrix for each regular cusp. Then $\mathcal{B}_{\mathcal{G}, k}$ is a basis for the space of Eisenstein series $\mathcal{E}_{k}(\mathcal{G})$ \cite[4.2]{myself}. For a subfield $\mathbb{K}$ of $\mathbb{C}$ we write $\mathcal{E}_{k}(\mathcal{G}, \mathbb{K})$ for the $\mathbb{K}$ span of $\mathcal{B}_{\mathcal{G},k}$. Recall that if $\Gamma(N)$ is a prinicipal level contained in $\mathcal{G}$ then $\mathcal{E}_{k}(\mathcal{G}) = \mathcal{E}_k\big(\Gamma(N)\big)^{\mathcal{G}}$. One of the key features of our treatment in \cite{myself} is that we explicitly express each spectral Eisenstein series on $\mathcal{G}$ as a sum of the spectral Eisenstein series on $\Gamma(N)$. As a consequence, the $\mathbb{K}$-structure on the space of Eisenstein series is compatible with restriction to a principal level. In more detail, if $\Gamma(N) \subseteq \mathcal{G}$ then $\mathcal{E}_k(\mathcal{G}, \mathbb{K}) = \mathcal{E}_k\big(\Gamma(N), \mathbb{K}\big)^{\mathcal{G}}$. 
To write down uniform statements for all weights $\geq 2$, we introduce an \textit{extended} space of Eisenstein series 
\[\mathcal{E}^{\ast}_k(\mathcal{G}, \mathbb{K}) = \begin{cases}
	\text{$\mathbb{K}$-span of $\{E_{k,x} \mid x \in \mathcal{C}(\mathcal{G})\}$,} & \text{if $k = 2$;}\\
	\mathcal{E}_k(\mathcal{G}, \mathbb{K}), & \text{if $k \geq 3$.}
\end{cases} \]
It is clear that $\mathcal{E}^{\ast}_2(\mathcal{G}, \mathbb{K}) \cap \text{Hol}(\mathbb{H}) = \mathcal{E}_2(\mathcal{G}, \mathbb{K})$ and $\mathcal{E}_2(\mathcal{G}, \mathbb{K})$ is a subspace of codimension $1$ inside $\mathcal{E}^{\ast}_2(\mathcal{G}, \mathbb{K})$. The $\mathbb{K}$-rational structure on the extended space of weight $2$ is also compatible with restriction to a principal level in the sense described above.  

\subsection{Eisenstein series on the universal elliptic curve}
\label{section3.2}
The following discussion aims to study a special family of Eisenstein series arising from the $N$-torsion points of the universal elliptic curve that exhibits remarkable rationality properties regarding period polynomials. 

Let $k$ be an integer $\geq 2$ and $\Lambda = \mathbb{Z}^2$. With each $\bar{\lambda} \in \Lambda/N\Lambda$ one associates the Hecke Eisenstein series described as   
\[G_{k}(\tau, \bar{\lambda}, N; s) := \sum_{\substack{(c,d) \equiv \lambda (N),\\ (c,d) \neq (0,0)}} \frac{v(\tau)^s}{(c\tau+d)^k \lvert c\tau+d \rvert^{2s}}\hspace{.3cm}(\substack{\tau \in \mathbb{H}})\]
where $\lambda$ is lift of $\bar{\lambda}$ and $v(\tau)$ is the imaginary part of $\tau$. This series converges uniformly on the compact subsets of $\text{Re}(k+2s) > 2$ to define an analytic function of $s$ in this region that admits an analytic continuation to the whole $s$ plane. Set $G_{k}(\tau, \bar{\lambda}, N) := G_{k}(\tau, \bar{\lambda}, N; 0)$. As a function on $\mathbb{H}$ the Fourier expansion of the $G$-series is 
$G_{k}(\tau, \bar{\lambda}, N) = \sum_{j \geq 0} A_{j}(\bar{\lambda}) \mathbf{e}(\frac{j\tau}{N}) + C_{0}(\tau)$ where 
\begin{equation}
\label{3.3}
\begin{gathered}
A_{0}(\bar{\lambda}) := \mathbbm{1}_{\mathbb{Z}/N\mathbb{Z}}(\bar{\lambda}_1, 0) \sum_{\substack{n \in \mathbb{Z} - \{0\}, \\ n \equiv \lambda_2(N)}} \frac{1}{n^k},\;C_0(\tau) := -\mathbbm{1}_{\mathbb{Z}}(k,2)\frac{\pi}{N^2v(\tau)}\\
A_{j}(\bar{\lambda}) := \frac{(-2\pi i)^k}{(k-1)! N^k} \sum_{\substack{n \in \mathbb{Z} - \{0\},\\ n \mid j, \frac{j}{n}\equiv \lambda_1(N)}} \text{sgn}(n)n^{k-1}\mathbf{e}(\frac{n\lambda_2}{N}) \hspace{.4cm}(\substack{j \geq 1})
\end{gathered}
\end{equation}
and $(\lambda_1, \lambda_2)$ is a lift of $\bar{\lambda}$ to $\Lambda$ \cite[9.1]{shimura07}. Recall the rational structure on the space of Eisenstein series from Section~\ref{section3.1} arising from the spectral basis. Suppose that $\mathbb{K}$ is a subfield of $\mathbb{C}$ containing $\mathbb{Q}(\mu_N)$. Then the $\mathbb{K}$-span of $\{G_{k}(\bar{\lambda}, N) \mid \bar{\lambda} \in \Lambda/N\Lambda\}$ equals $(2\pi i)^k \mathcal{E}_{k}^{\ast}\big(\Gamma(N), \mathbb{K}\big)$ and its intersection with $\text{Hol}(\mathbb{H})$ is $(2\pi i)^k \mathcal{E}_{k}\big(\Gamma(N), \mathbb{K}\big)$ \cite[5.2]{myself}. The explicit Fourier expansion makes $G$-series accessible for the computation of $L$-function and $L$-values. A standard approach towards this problem uses a cyclotomic linear combination of $G$-series to write down a Hecke eigenfunction on $\Gamma(N)$ whose $L$-function is given by a product of two Dirichlet $L$-functions \cite[IV-39]{ogg}. However, it is difficult to study the rationality properties of such product $L$-functions. 

We define the Eisenstein series on the universal elliptic curve as 
\[e_{k}: \mathbb{C} \times \mathbb{H} \to \mathbb{C}, \hspace{.3cm} e_{k}(\eta, \tau) = \sum_{w \in \Lambda_{\tau} - \{0\}} \frac{\chi_{\eta} (w)}{w^k}\hspace{.5cm}(\substack{k \geq 3})\]  
where $\Lambda_{\tau} = \{(c\tau+d) \mid (c,d) \in \Lambda\}$ and $\chi_{\eta} (w) = \mathbf{e}(\frac{w\bar{\eta} - \bar{w}\eta}{\tau - \bar{\tau}})$ \cite[4]{levin}. This series converges uniformly on compact subsets to define a real analytic function that transforms like a Jacobi form of weight $k$ and index $0$. In particular $e_{k}(\eta + c\tau + d, \tau) = e_{k}(\eta,\tau)$ for each $(c,d) \in \Lambda$. With $\bar{\lambda} \in \Lambda/N\Lambda$ one associates the division value  
\[e_{k}(\tau, \bar{\lambda}, N) := e_{k}(\frac{\lambda_1 \tau}{N} + \frac{\lambda_2}{N}, \tau) = \sum_{\substack{(c,d) \in \Lambda - \{0\}}} \frac{\mathbf{e}(\frac{c\lambda_2 - d\lambda_1}{N})}{(c\tau+d)^k}; \hspace{.3cm}\bar{\lambda} = \overline{(\lambda_1, \lambda_2)}.\]
Let $\mathbf{b} : \Lambda/N\Lambda \times \Lambda/N\Lambda \to \mathbb{C}$ be the nondegenerate bilinear form defined by \[\mathbf{b}(\bar{\lambda}, \bar{\theta}) = \mathbf{e}(\frac{\theta_1\lambda_2 - \theta_2\lambda_1}{N}); \hspace{.3cm}\text{$\bar{\lambda} = \overline{(\lambda_1, \lambda_2)}$ and $\bar{\theta} = \overline{(\theta_1, \theta_2)}$}.\] Then the relation between $e$-series and $G$-series is as follows:  
\begin{equation}
\label{3.4new}
e_k(\bar{\lambda}, N) = \sum_{\bar{\theta} \in \Lambda/N\Lambda} \mathbf{b}(\bar{\lambda}, \bar{\theta})G_{k}(\bar{\theta},N). 
\end{equation}
Let $k$ be an integer $\geq 2$ and $\bar{\lambda} \in \Lambda/N\Lambda$. We define the \textit{elliptic} Eisenstein series of weight $k$ and parameter $\bar{\lambda}$, denoted $e_{k}(\bar{\lambda}, N)$, using the identity \eqref{3.4new}. Note that $\mathbf{b}$ is compatible with the natural right action of $\Gamma$ on $\Lambda$, i.e., $\mathbf{b}(\bar{\lambda}\gamma, \bar{\theta}\gamma) = \mathbf{b}(\bar{\lambda}, \bar{\theta})$ for each $\gamma \in \Gamma$. Thus $e_{k}(\bar{\lambda},N)\lvert_{\gamma} = e_{k}(\bar{\lambda}\gamma, N)$. 
In particular $e_{k}(\bar{\lambda},N)$ is invariant under $\Gamma(N)$. One uses the orthogonality of characters for $\Lambda/N\Lambda$ to invert the linear relation \eqref{3.4new} and discovers that   \[G_{k}(\bar{\theta},N) = \frac{1}{N^2}\sum_{\bar{\lambda} \in \Lambda/N\Lambda} \mathbf{b}(\bar{\theta}, \bar{\lambda})e_{k}(\bar{\lambda},N).\] 
Therefore the $\mathbb{K}$-span of $\{e_{k}(\bar{\lambda}, N) \mid \bar{\lambda} \in \Lambda/N\Lambda\}$ equals the $\mathbb{K}$-span of $\{G_{k}(\bar{\lambda}, N) \mid \bar{\lambda} \in \Lambda/N\Lambda\}$ whenever $\mathbb{Q}(\mu_N) \subseteq \mathbb{K}$. 

We next record the Fourier expansion formula for the elliptic Eisenstein series. Let $B_{\bullet}(-)$ be the Bernoulli polynomial defined by the generating function identity \[\frac{X\exp(tX)}{\exp(X) - 1} = \sum_{n \geq 0} \frac{B_{n}(t)}{n!}X^n.\] 
Our calculation requires the Fourier expansion formula for the Bernoulli polynomials \cite[4.5]{shimura07}:
\begin{equation}
\label{3.5new}
B_{k}(t) = -\frac{k!}{(2\pi i)^k} \sum_{n \in \mathbb{Z} - \{0\}} \frac{\mathbf{e}(nt)}{n^k}. \hspace{.5cm}(\substack{k \geq 2,\; 0 \leq t \leq 1})
\end{equation}  
The formula is valid even for $k = 1$ if $0 < t < 1$ and one interprets the sum as $\lim_{m \to \infty} \sum_{0 < \lvert n \rvert \leq m}$.

\begin{lemma}
\label{lemma3.1}
Let $k$ be an integer $\geq 2$ and $\bar{\lambda} \in \Lambda/N\Lambda$. Suppose that $(\lambda_1, \lambda_2)$ is a lift of $\bar{\lambda}$ that satisfies $0 \leq \lambda_1 < N$. The Fourier series for $e_{k}(\bar{\lambda}, N)$ is 
\begin{equation*}
\begin{aligned}
e_{k}(\bar{\lambda},N) & = - \frac{(-2\pi i)^{k}}{k!} B_{k}(\frac{\lambda_1}{N}) - \mathbbm{1}_{\mathbb{Z}}(k,2)\mathbbm{1}_{\Lambda/N\Lambda}(\bar{\lambda}, 0) \frac{\pi}{v(\tau)}\\
& + \frac{(-2\pi i)^k}{(k-1)!N^{k-1}} \sum_{\substack{m, n \geq 1,\\ n \equiv \lambda_1 (N)}} n^{k-1} \mathbf{e}(\frac{mn\tau}{N}+ \frac{m\lambda_2}{N})\\
& + \frac{(2\pi i)^k}{(k-1)!N^{k-1}} \sum_{\substack{m, n \geq 1,\\ n \equiv -\lambda_1 (N)}} n^{k-1} \mathbf{e}(\frac{mn\tau}{N} - \frac{m\lambda_2}{N}). 
\end{aligned}
\end{equation*} 
\end{lemma}

\begin{proof}
We substitute the Fourier expansion formula for the $G$-series \eqref{3.3} into the defining relation \eqref{3.4new} and simplify the expression to obtain the identity above. The constant term is a consequence of \eqref{3.5new} while the real analytic term follows from a straightforward calculation. 
\end{proof}

\begin{corollary}
\label{corollary3.2}
Let $\bar{\lambda} \in \Lambda/N\Lambda$. Then $e_{2}(\bar{\lambda}, N)$ is holomorphic if and only if $\bar{\lambda} \neq 0$. Moreover $\{e_2(\bar{\lambda},N) \mid \bar{\lambda} \neq 0\}$ is a spanning set for the holomorphic subspace $(2\pi i)^{2}\mathcal{E}_2\big(\Gamma(N), \mathbb{K}\big)$ whenever $\mathbb{Q}(\mu_N) \subseteq \mathbb{K}$.  
\end{corollary}

\begin{proof}
The first part of the assertion is obvious. The extended space of $G$-series $(2\pi i)^2\mathcal{E}^{\ast}_{2}\big(\Gamma(N),\mathbb{K}\big)$ contains $(2\pi i)^2\mathcal{E}_{2}\big(\Gamma(N),\mathbb{K}\big)$ as a subspace of codimension $1$. But $\{e_2(\bar{\lambda},N) \mid \bar{\lambda} \neq 0\}$ is already a subset of holomorphic subspace. Therefore this collection spans $(2\pi i)^2\mathcal{E}_{2}\big(\Gamma(N),\mathbb{K}\big)$.   
\end{proof}

For convenience, we introduce the notation 
\begin{equation*}
	\mathcal{I}_{N,k} := \begin{cases}
		\Lambda/N\Lambda, & \text{$k$ even and $\geq 4$;}\\
		\Lambda/N\Lambda, & \text{$k$ odd and $N \geq 3$;}\\
		\Lambda/N\Lambda - \{0\}, & k=2;\\
		\emptyset, & \text{otherwise.}
	\end{cases}	
\end{equation*}
Note that $\mathcal{I}_{N,k}$ is stable under the action of $\Gamma$ on $\Lambda/N\Lambda$. 

\subsection{Fourier expansion formulas}
\label{section3.3}
Let $k$ and $\ell$ be two integers with $w := k + \ell \geq 3$. Suppose that $N \geq 1$ and $\bar{\lambda} \in \Lambda/N\Lambda$. Then the Maass Eisenstein series \cite{maass} of index $(k, \ell)$ attached to the parameter $\bar{\lambda}$ is given by 
\begin{equation}
\label{3.6}
G_{k,\ell} (\tau, \bar{\lambda}, N) := \sum_{\substack{(c,d) \equiv \lambda (N),\\ (c,d) \neq (0,0)}} \frac{1}{(c\tau+d)^k (c\bar{\tau}+d)^{\ell}}.
\end{equation} 
In the spectral theory of Eisenstein series, one writes down the Fourier expansion of $G_{k,\ell}$ using Whittaker functions \cite[p.134]{shimura07}. Recently, Brown \cite{brown-nonhol} has described a class of nonholomorphic modular forms on $\text{SL}_2(\mathbb{Z})$ with Fourier expansion in $\mathbb{C}[[q, \bar{q}]][v(\tau), \frac{1}{v(\tau)}]$ where $q = \mathbf{e}(\tau)$. In view of this theory, it is interesting to ask if the series $G_{k,\ell}$ has Fourier expansion in $\mathbb{C}[[q_N, \bar{q}_N]][v(\tau), \frac{1}{v(\tau)}]$ where $q_N = \mathbf{e}(\frac{\tau}{N})$. It is known \cite{o'sullivan} how to deduce this result from the spectral Fourier expansion, at least, for $N = 1$. The current discussion aims to show how to directly write down a Fourier expansion for $G_{k,\ell}$ of the form $\mathbb{C}[[q_N, \bar{q}_N]][v(\tau), \frac{1}{v(\tau)}]$ without any reference to the spectral expansion. This formula plays a role in the construction of generalized invariants in Section~\ref{section5}. Our calculation relies on the following key lemma:

\begin{lemma}
\label{lemma3.3}\emph{\cite[A3]{shimura07}} Let $k$ and $\ell$ be two integers with $k + \ell \geq 2$. Suppose that $\tau = u + i v \in \mathbb{H}$. Then  
\begin{equation}
\label{3.7new}
\begin{aligned}
\sum_{m \in \mathbb{Z}} \frac{1}{(\tau + m)^k (\bar{\tau} + m)^{\ell}} & = (- 2\pi i) \sum_{\substack{0 \leq r \leq k-1,\\ n \geq 0}} \frac{\binom{-\ell}{r}(-2\pi i n)^{k- 1-r}\mathbf{e}(n\tau)}{(-2iv)^{\ell + r}(k - 1- r)!} \\
& \hspace{.7cm} + 2\pi i \sum_{\substack{0 \leq r \leq \ell-1,\\ n < 0}} \frac{\binom{-k}{r}(-2\pi i n)^{\ell- 1-r}\mathbf{e}(n\bar{\tau})}{(2iv)^{k + r}(\ell - 1- r)!}
\end{aligned}
\end{equation}
where $\binom{\bullet}{\bullet}$ is the formal binomial coeffcient in Section~\ref{section1.1}.  
\end{lemma} 

\begin{proof}
For a fixed $v > 0$, write $f_v(u) = (u+iv)^{-k}(u-iv)^{-\ell}$. Now, an application of the Poisson summation formula shows that 
\[\sum_{m \in \mathbb{Z}} \frac{1}{(\tau + m)^k (\bar{\tau} + m)^{\ell}} = \sum_{m \in \mathbb{Z}} f_{v}(u+m) = \sum_{n \in \mathbb{Z}} \widehat{f}_{v}(n)\mathbf{e}(nu)\]
where $\widehat{f}_{v}(n) = \int_{\mathbb{R}} \frac{\mathbf{e}(-nt)dt} {(t + iv)^{k}(t - iv)^{\ell}}$. We evaluate $\widehat{f}_{v}(n)$ using contour integration on the complex plane opposed to the theory of special functions. Set $\varphi_n(z) = \frac{\mathbf{e}(-nz)}{(z+iv)^k(z-iv)^{\ell}}$. For $R$ be a positive real number, let $I_R$ be the segment joining $-R$ to $R$, $C^{+}_{R}$ be the semicircular arc of radius $R$ on the upper half-plane joining $R$ to $-R$, and $C^{-}_{R}$ be the semicircular arc of radius $R$ on the lower half-plane joining $-R$ to $R$. Letting $R \to \infty$, we discover that $\int_{C_{R}^{-}} \lvert \varphi_{n}(z) \rvert dz \to 0$ if $n \geq 0$ and $\int_{C_{R}^{+}} \lvert \varphi_{n}(z) \rvert dz \to 0$ if $n < 0$. Moreover, the  residue of $\varphi_n(z)$ at $-iv$, resp. $iv$, equals 
\[\text{$\sum_{r=0}^{k-1} \frac{\binom{-\ell}{r} (-2\pi in)^{k-1-r} \mathbf{e}(niv)} {(-2iv)^{\ell+r}(k-1-r)!}$, resp. $\sum_{r=0}^{\ell-1} \frac{\binom{-k}{r} (-2\pi in)^{\ell-1-r} \mathbf{e}(-niv)} {(2iv)^{k+r}(\ell-1-r)!}$.}\]
We evaluate $\widehat{f}_v(n)$ using the closed contour defined by $C_R^{-}$ and $I_R$ if $n \geq 0$ and $C_R^{+}$ and $I_R$ if $n < 0$ for large $R$ and let $R \to \infty$ to arrive at the desired identity.    
\end{proof}

Lemma~\ref{lemma3.3} immediately yields the Fourier expansion formula of the Maass Eisenstein series.

\begin{theorem}
\label{theorem3.4}
Let the notation be the same as \eqref{3.6}. Suppose that $(\lambda_1, \lambda_2)$ is a lift of $\bar{\lambda}$ that satisfies $0 \leq \lambda_1 < N$. Then 
\[G_{k, \ell}(\tau, \bar{\lambda}, N) = A_0(\bar{\lambda}) + C_0(\bar{\lambda}, \tau) + \sum_{j \geq 1} A_j(\bar{\lambda}, \tau)q_N^j + \sum_{j \geq 1} C_j(\bar{\lambda}, \tau)\bar{q}_N^j\]
where 
\begin{equation*}
\begin{gathered}
A_0 = \mathbbm{1}_{\mathbb{Z}/N\mathbb{Z}}(\bar{\lambda}_1,0) \sum_{\substack{n \equiv \lambda_2(N),\\ n \neq 0}} \frac{1}{n^w},\\ 
C_0 = -2\pi i\frac{\Big(\binom{-\ell}{k-1}\zeta^{\ast}(\frac{\lambda_1}{N}, w-1) + \binom{-k}{\ell-1}\zeta(1-\frac{\lambda_1}{N}, w-1)\Big)}{N^{w}(-2iv)^{w-1}},\\
A_{j} = -2\pi i \sum_{\substack{mn = j, \\ m \equiv \lambda_1(N)}} \sum_{0 \leq r \leq k-1} \frac{\sgn(n)}{N^{k-r}} \frac{\binom{-\ell}{r} (-2\pi i n)^{k-1-r}\mathbf{e}(\frac{n\lambda_2}{N})} {(-2imv)^{\ell +r}(k-1-r)!},\hspace{.3cm}(\substack{j \geq 1})\\
C_j = 2 \pi i \sum_{\substack{mn = j, \\ m \equiv \lambda_1(N)}} \sum_{0 \leq r \leq \ell-1} \frac{\sgn(n)}{N^{\ell-r}} \frac{\binom{-k}{r} (2\pi i n)^{\ell-1-r}\mathbf{e}(-\frac{n\lambda_2}{N})} {(2imv)^{k +r}(\ell-1-r)!}\hspace{.3cm}(\substack{j \geq 1})
\end{gathered}
\end{equation*}
and 
\[\zeta^{\ast}(\frac{\lambda_1}{N}, s)  = \begin{cases}
	\zeta(\frac{\lambda_1}{N}, s), & \text{if $0 < \lambda_1 < N$;}	\\
	\zeta(1, s), & \text{if $\lambda_1 = 0$.}
\end{cases}\]
In particular, $G_{k, \ell}(\bar{\lambda};N) \in \mathbb{C}[[q_N, \bar{q}_N]][v(\tau), \frac{1}{v(\tau)}]$. 
\end{theorem}

\begin{proof}
The assertion is an easy consequence of the formula \eqref{3.7new}. The terms in \eqref{3.6} with $c = 0$ give rise to $A_0$ while the terms in \eqref{3.7new}  with $n = 0$ amount to $C_0$ in the identity above. 
\end{proof}

With the notation of \eqref{3.6}, one can also define a nonholomorphic elliptic Eisenstein series of index $(k, \ell)$ by 
\begin{equation}
\label{3.8new}
e_{k,\ell}(\bar{\lambda}, N) := \sum_{(c,d) \neq (0,0)} \frac{\mathbf{b}(\bar{\lambda}, \overline{(c,d)})} {(c\tau + d)^k (c\bar{\tau} + d)^\ell};
\end{equation}
cf. \cite[4]{levin}.   
The following Fourier expansion formula is necessary for us: 

\begin{corollary}
\label{corollary3.5} Suppose that $(\lambda_1, \lambda_2) \in \Lambda$ is a lift of $\bar{\lambda}$ with $0 \leq \lambda_1 < N$. Then 
\[e_{k, \ell}(\tau, \bar{\lambda}, N) = A^{\emph{ell}}_0(\bar{\lambda}) + C^{\emph{ell}}_0(\bar{\lambda}, \tau) + \sum_{j \geq 1} A^{\emph{ell}}_j(\bar{\lambda}, \tau)q_N^j + \sum_{j \geq 1} C^{\emph{ell}}_j(\bar{\lambda}, \tau)\bar{q}_N^j\]
where 
\begin{equation*}
\begin{gathered}
A^{\emph{ell}}_0 = - \frac{(-2 \pi i)^w}{w!}B_{w}(\frac{\lambda_1}{N}),\\ 
C^{\emph{ell}}_0 = -2\pi i\mathbbm{1}_{\mathbb{Z}/N\mathbb{Z}}(\bar{\lambda}_1, 0)\frac{\Big(\binom{-\ell}{k-1}\emph{Li}_{w-1}(\mathbf{e}(\frac{\lambda_2}{N})) + \binom{-k}{\ell-1}\emph{Li}_{w-1}\big(\mathbf{e}(-\frac{\lambda_2}{N})\big)}{(-2iv)^{w-1}},\\
A^{\emph{ell}}_{j} = -2\pi i \sum_{\substack{mn = j, \\ n \equiv \lambda_1(N)}} \sum_{0 \leq r \leq k-1} \frac{\sgn(n)}{N^{k-r-1}} \frac{\binom{-\ell}{r} (-2\pi i n)^{k-1-r}\mathbf{e}(\frac{m\lambda_2}{N})} {(-2imv)^{\ell +r}(k-1-r)!},\hspace{.3cm}(\substack{j \geq 1})\\
C^{\emph{ell}}_j = 2 \pi i \sum_{\substack{mn = j, \\ n \equiv -\lambda_1(N)}} \sum_{0 \leq r \leq \ell-1} \frac{\sgn(n)}{N^{\ell-r-1}} \frac{\binom{-k}{r} (2\pi i n)^{\ell-1-r}\mathbf{e}(\frac{m\lambda_2}{N})} {(2imv)^{k +r}(\ell-1-r)!}\hspace{.3cm}(\substack{j \geq 1})
\end{gathered}
\end{equation*}
\end{corollary}

\begin{proof}
Follows from Theorem~\ref{theorem3.4} and the identity \[e_{k, \ell}(\bar{\lambda}, N) = \sum_{\bar{\theta} \in \Lambda/N\Lambda} \mathbf{b}(\bar{\lambda}, \bar{\theta})G_{k, \ell}(\bar{\theta},N).\]  
\end{proof}

Next, we connect Theorem~\ref{theorem3.4} with the spectral Eisenstein series in Section~\ref{section3.1}. Let the notation be as in \eqref{3.1} and suppose that $\mathcal{G}$ contains the principal subgroup $\Gamma(N)$. 

\begin{corollary}
\label{corollary3.6}
Let $s_0$ be an integer so that $\emph{Re}(k + 2s_0) \geq 3$. Then 
\[E_{k,x}(\tau, s_0) \in \mathbb{C}[[q_N, \bar{q}_N]][v(\tau), \frac{1}{v(\tau)}].\]
\end{corollary}

\begin{proof}
Recall that $E_{k,x}(\tau, s_0)$ is a linear combination of the spectral Eisenstein series on $\Gamma(N)$. Moreover, each spectral Eisenstein series on $\Gamma(N)$ is an explicit linear combination of Hecke's Eisenstein series on $\Gamma(N)$; see \cite[5]{myself}. Therefore, the assertion is an easy consequence of the theorem. 
\end{proof}

\begin{remark}
\label{remark3.7} One can use the above technique to write down Fourier expansion for the suitably regularized nonholomorphic Eisenstein series of weight $1, 2$ and verify that they lie in $\mathbb{C}[[q_N, \bar{q}_N]][v(\tau), \frac{1}{v(\tau)}]$. In terms of singularities, the elliptic series \eqref{3.8new} behave somewhat better that the Maass Eisenstein series exactly like the holomorphic case in Section~\ref{section3.2}. 
\end{remark}

\section{Rationality of periods}
\label{section4}
\subsection{$L$-function and $L$-values of the elliptic Eisenstein series}
\label{section4.1}
Our next job is explicitly writing down the completed $L$-function \eqref{2.7} attached to the elliptic Eisenstein series. Let $x \in \mathbb{R}$ and $0 < a \leq 1$. One defines the Lerch zeta function \cite{apostol} by
\begin{equation}
\label{4.1}
\phi(x, a, s) = \sum_{n \geq 0} \frac{\mathbf{e}(nx)}{(n+a)^s}.
\end{equation}
The series above converges uniformly on the compact subsets of $\{s \mid \text{Re}(s) > 1\}$ to define an analytic function in this region. The Lerch series admits a meromorphic continuation to the whole $s$-plane with a simple pole at $s = 1$ whenever $x$ is an integer. If $x$ is not an integer, then $\phi(x, a, \cdot)$ extends to an entire function on the $s$-plane. Our treatment requires two special avatars of the Lerch function. The Lerch series reduces to the familiar Hurwitz zeta function $\zeta(a,s)$ for $x = 0$, i.e., $\zeta(a,s) = \phi(0,a,s)$. The value of the Hurwitz zeta function at nonpositive integers \cite[12.11]{apostolbook} is given by  
\begin{equation}
\label{4.2}
\zeta(a, -n) = - \frac{B_{n+1}(a)}{n+1}. \hspace{.3cm}(\substack{0 < a \leq 1,\; n \geq 0})
\end{equation}	
On the other hand setting $a = 1$ in \eqref{4.1} yields the polylogarithm function given by the convergent series \[\text{Li}_{s}\big(\mathbf{e}(x)\big) := \sum_{n \geq 1} \frac{\mathbf{e}(nx)}{n^s} = \mathbf{e}(x) \phi(x,1,s)\] in the region $\text{Re}(s) > 1$. The meromorphic continuation of the Lerch function provides a meromorphic continuation for $\text{Li}$. If $x$ is not an integer then $\text{Li}_1\big(\mathbf{e}(x)\big)$ equals the conditionally convergent logarithmic series $\sum_{n \geq 1} \frac{\mathbf{e}(nx)}{n}$. The Fourier expansion formula for the Bernoulli polynomials \eqref{3.5new} shows that 
\begin{equation}
\label{4.3}
\text{Li}_{k}\big(\mathbf{e}(x)\big) + (-1)^k \text{Li}_{k}\big(\mathbf{e}(-x)\big) = - \frac{(2\pi i)^k}{k!} B_{k}(x)\hspace{.3cm}(\substack{0 \leq x \leq 1})
\end{equation}
for each $k \geq 2$. The identity above is valid even for $k = 1$ if $0 < x < 1$.

\begin{lemma}
\label{lemma4.1}
Let the notation be as in Section~\ref{section3.2}. Suppose that $k \geq 2$ and $\bar{\lambda} \in \mathcal{I}_{N,k}$. Now suppose $(\lambda_1, \lambda_2)$ is a lift of $\bar{\lambda}$ that satisfies $0 \leq \lambda_1 < N$. Then
\begin{equation}
\label{4.4}
\begin{aligned}
& \frac{(2\pi)^s}{\Gamma(s)}L^{\ast}\big(e_{k}(\bar{\lambda},N),s\big)\\
& = \frac{(-2\pi i)^k}{(k-1)!}\big\{\emph{Li}_{s}\big(\mathbf{e}(\frac{\lambda_2}{N})\big) \zeta^{\ast}(\frac{\lambda_1}{N}, s-k+1) +\\
& \hspace{3cm}(-1)^k\emph{Li}_{s}\big(\mathbf{e}(-\frac{\lambda_2}{N})\big) \zeta(1-\frac{\lambda_1}{N}, s-k+1)\big\}
\end{aligned}
\end{equation}
where $\zeta^{\ast}(\cdot, \cdot)$ is as in Theorem~\ref{theorem3.4}.  
\end{lemma}     

\begin{proof}
We first assume that $\text{Re}(s) > k$. Then the completed $L$-function admits a convergent integral representation and the identity \eqref{4.4} is a straightforward consequence of Lemma~\ref{lemma3.1}. Observe that both sides of \eqref{4.4} meromorphically continue to the whole plane. Thus, the identity holds due to the principle of analytic continuation.  
\end{proof}

\begin{proposition}
\label{proposition4.2}
Let the notation be as in Lemma~\ref{lemma4.1}. We further assume that $0 \leq \lambda_2 < N$. Suppose that $1 \leq r \leq k-1$. Then 
\begin{equation*}
\begin{aligned}
L^{\ast}\big(e_{k}(\bar{\lambda},N), r\big) & = i^{r}\frac{(-2\pi i)^k}{(k-1)!} \frac{B_{k-r}(\frac{\lambda_1}{N})B_{r}(\frac{\lambda_2}{N})}{r(k-r)}\\
& + (-1)^{k-1} i^{k} \mathbbm{1}_{\mathbb{Z}/N\mathbb{Z}}(\lambda_1, 0)\mathbbm{1}_{\mathbb{Z}}(r,k-1)\frac{2\pi}{k-1}\emph{Li}_{k-1}\big(\mathbf{e}(\frac{\lambda_2}{N})\big)\\
& -\mathbbm{1}_{\mathbb{Z}/N\mathbb{Z}}(\lambda_2,0)\mathbbm{1}_{\mathbb{Z}}(r,1)\frac{2\pi}{k-1}\emph{Li}_{k-1}\big(\mathbf{e}(-\frac{\lambda_1}{N})\big). 
\end{aligned}
\end{equation*}  
\end{proposition}

The proof of the identity involves the following standard properties of the Bernoulli polynomials \cite[4.5]{shimura07}: 
\begin{equation}
\label{4.5}
\begin{gathered}
B_{n}(1-t) = (-1)^nB_n(t), \hspace{.3cm}(\substack{0 \leq t \leq 1,\; n \geq 0})\\
B_n(1+t) = B_n(t) + nt^{n-1}.\hspace{.3cm}(\substack{0 \leq t \leq 1,\; n \geq 1})
\end{gathered}
\end{equation}
In particular \[B_{n}(1) = \begin{cases}
	B_n(0), & \text{if $n \geq 2$;}\\
	B_n(0)+1, &\text{if $n=1$.}
\end{cases}\] 

\begin{proof}
If $0 < \lambda_1, \lambda_2 < N$ then the identity is an easy consequence of the formulas given in \eqref{4.2}-\eqref{4.5}:
\begin{align*}
& L^{\ast}\big(e_k(\bar{\lambda},N), r\big)\\
& = - \frac{(r-1)!}{(2\pi)^r} \frac{(-2\pi i)^k}{(k-1)!(k-r)} \Big\{\text{Li}_r\big(\mathbf{e}(\frac{\lambda_2}{N})\big) B_{k-r}(\frac{\lambda_1}{N}) +\\ 
& \hspace{6cm}(-1)^k \text{Li}_r\big(\mathbf{e}(-\frac{\lambda_2}{N})\big) B_{k-r}(1-\frac{\lambda_1}{N})\Big\}\\
& = i^r \frac{(-2\pi i)^k}{(k-1)!} \frac{B_{k-r}(\frac{\lambda_1}{N})B_r(\frac{\lambda_2}{N})}{r(k-r)}. 
\end{align*}   
The extremal case of \eqref{4.5} together with \eqref{4.3} guarantee that the calculation for $0 < \lambda_1, \lambda_2 < N$ extends to the following situations without any difficulty: 
\begin{itemize}
\item all $(\lambda_1, \lambda_2)$ if $2 \leq r \leq k-2$,
\item $\{(0, \lambda_2) \mid \lambda_2 \neq 0\}$ if $r = 1$, $\{(\lambda_1, 0) \mid \lambda_1 \neq 0\}$ if $r = k- 1$ for $k \geq 3$.
\end{itemize} 
Suppose that $r = k-1$ and $\lambda_1 = 0$. The hypothesis on $\bar{\lambda}$ implies that if $k = 2$ then $\lambda_2$ is automatically nonzero. In this situation $\zeta^{\ast}$ contributes $B_1(1)$ instead of $B_{1}(0)$ and this discrepancy gives rise to the desired term involving $\text{Li}_{k-1}\big(\mathbf{e}(\frac{\lambda_2}{N})\big)$. It remains to verify the identity for $r = 1$ and $\lambda_2 = 0$. As before, if $k = 2$ then $\lambda_1 \neq 0$. The functional equation for the completed $L$-function implies that 
\begin{align*}
& L^{\ast}\big(e_{k}\big(\overline{(\lambda_1, 0)}, N\big),1\big)\\
& = (-i)^kL^{\ast}\big(e_{k}\big(\overline{(0, \lambda_1)},N\big),k-1\big)\\
& =  i \frac{(-2\pi i)^k}{(k-1)!}\frac{B_{k-1}(\frac{\lambda_1}{N}) B_1(0)}{k-1} - \frac{2\pi}{k-1} \text{Li}_{k-1}\big(\mathbf{e}(-\frac{\lambda_1}{N})\big). 
\end{align*}
Here, in the last step, we are using \eqref{4.3} to simplify the expression.  
\end{proof}

\subsection{Descent of cohomology classes}
\label{section4.2}
Let the notation be as in Section~\ref{section3.2}.  The explicit expressions for the special values of the completed $L$ function in Proposition~\ref{proposition4.2} motivates us to normalize the elliptic Eisenstein series as 
\begin{equation*}
\mathfrak{e}_{k}(\bar{\lambda},N):= \frac{1}{(-2\pi i)^k} e_{k}(\bar{\lambda},N). \hspace{.3cm}(\substack{\bar{\lambda} \in \Lambda/N\Lambda})
\end{equation*}
The following lemma computes the value of the equivariant period function \eqref{2.1} attached to the base point at infinity for $\Gamma$-stable subset $\mathcal{M}_k\big(\Gamma(N)\big)$ at the normalized elliptic Eisenstein series. 

\begin{lemma}
\label{lemma4.3}
Let $k \geq 2$ and $\bar{\lambda} \in \mathcal{I}_{N,k}$. Suppose that $(\lambda_1, \lambda_2)$ is a lift of $\bar{\lambda}$ satisfying $0 \leq \lambda_1, \lambda_2 < N$. Then 
\begin{align*}
& \mathsf{p}_{k}^{\ast}(T) \big(\mathfrak{e}_k(\bar{\lambda},N)\big)  = - \frac{B_k(\frac{\lambda_1}{N})}{k!(k-1)}\big((X+1)^{k-1} - X^{k-1}\big),\\
& \mathsf{p}_{k}^{\ast}(S) \big(\mathfrak{e}_k(\bar{\lambda},N)\big)  = - \sum_{r=0}^{k-2} \binom{k}{r+1} \frac{B_{k-r-1}(\frac{\lambda_1}{N})B_{r+1}(\frac{\lambda_2}{N})}{k!(k-1)}X^{k-2-r} + (-1)^{k-1}\\
& \mathbbm{1}_{\mathbb{Z}/N\mathbb{Z}}(\lambda_1, 0)\frac{\emph{Li}_{k-1}\big(\mathbf{e}(\frac{\lambda_2}{N})\big)}{(k-1)(-2\pi i)^{k-1}} + \mathbbm{1}_{\mathbb{Z}/N\mathbb{Z}}(\lambda_2, 0)\frac{\emph{Li}_{k-1}\big(\mathbf{e}(-\frac{\lambda_1}{N})\big)}{(k-1)(-2\pi i)^{k-1}}X^{k-2}. 
\end{align*}	
\end{lemma}

\begin{proof}
The demonstration is a straightforward exercise using the formulas in Proposition~\ref{proposition2.13}, Lemma~\ref{lemma3.1}, and Proposition~\ref{proposition4.2}. 
\end{proof}

Observe that the value of the period polynomial at $S$ involves polylogarithm functions. Thus, it is difficult to examine the rationality properties of the period polynomials described above. We next explain how to change the Eichler-Shimura cocycle attached to a function by a coboundary so that the resulting cocycle has coefficients in $\mathbb{Q}$; cf. \cite[7]{zagier-gangl}. Let 
\begin{equation}
\label{4.6}
[\text{ES}_k] : \mathcal{M}_{k}\big(\Gamma(N)\big) \to H^{1}\big(\Gamma(N), V_{k-2,\mathbb{C}}\big)
\end{equation}  
be the Eichler-Shimura map attached to the space of modular forms.

\begin{proposition}
\label{proposition4.4new}
Let $k \geq 2$ and $\bar{\lambda} \in \mathcal{I}_{N,k}$. Then the image of $\mathfrak{e}_{k}(\bar{\lambda},N)$ under the Eichler-Shimura map \eqref{4.6} is defined over $\mathbb{Q}$.  
\end{proposition} 

\begin{proof}
We use the induced picture of Section~\ref{section2.3} to compute the Eichler-Shimura map. In the light of the explicit Shapiro map \eqref{2.5} it suffices to check that the image of $\mathfrak{e}_{k}(\bar{\lambda},N)$ under $\text{Ind-}[\text{ES}_k]$ admits a cocyle representative in $Z^{1}\big(\Gamma, \text{Fun}(\Gamma(N)\backslash \Gamma, V_{k-2,\mathbb{Q}})\big)$. For $\sigma \in \Gamma$ let $(\lambda_1^{\sigma}, \lambda_2^{\sigma})$ denote the unique lift of $\bar{\lambda}\sigma$ in $[0,N) \times [0,N)$. Note that $(\lambda_1^{\sigma}, \lambda_2^{\sigma})$ depends only on the coset of $\sigma$ in $\Gamma(N) \backslash \Gamma$. 
	
\textbf{Case I.} $k \geq 3$.
	
Define $F \in \text{Fun}\big(\Gamma(N) \backslash \Gamma, V_{k-2,\mathbb{C}}\big)$ by setting \[F\big(\Gamma(N)\sigma\big) = i^{3-k} L^{\ast}\big(\mathfrak{e}_{k}(\bar{\lambda}\sigma, N), k-1\big).\]
A straight-forward computation using the formulas in Proposition~\ref{proposition4.2} and Lemma~\ref{lemma4.3} shows that the modified cocycle
\begin{equation*}
\Gamma \to \text{Fun}\big(\Gamma(N)\backslash \Gamma, V_{k-2,\mathbb{C}}\big),\; \gamma \mapsto \text{Ind-}\text{ES}^{\ast}_{k}(\mathfrak{e}_k(\bar{\lambda},N))(\gamma) + F\lvert_{\gamma} - F, 
\end{equation*}
takes values in  $\text{Fun}\big(\Gamma(N)\backslash \Gamma, V_{k-2,\mathbb{Q}}\big)$ at $T$ and $S$.
	
\textbf{Case II.} $k=2$.
	
The problem, in this case, is that both the transcendental terms in the special $L$-value may simultaneously contribute a term to the calculation. We construct another function $F \in \text{Fun}\big(\Gamma(N) \backslash \Gamma, \mathbb{C}\big)$ by putting 
\[F\big(\Gamma(N)\sigma\big) = \begin{cases}
		0, & \text{if $\lambda_1^{\sigma} \neq 0$;}\\
		i L^{\ast}\big(\mathfrak{e}_2(\bar{\lambda}\sigma, N), 1\big), & \text{if $\lambda_1^{\sigma} = 0$.}
	\end{cases}\]  
Then the modified cocycle 
\begin{equation*}
\Gamma \to \text{Fun}\big(\Gamma(N)\backslash \Gamma, \mathbb{C}\big),\; \gamma \mapsto \text{Ind-}\text{ES}^{\ast}_{2}(\mathfrak{e}_2(\bar{\lambda},N))(\gamma) + F\lvert_{\gamma} - F, 
\end{equation*}
again takes values in $\text{Fun}\big(\Gamma(N)\backslash \Gamma, \mathbb{Q}\big)$ at $T$ and $S$.  
\end{proof}

Let $\mathbb{K}$ be a subfield of $\mathbb{C}$ containing $\mathbb{Q}(\mu_N)$. Recall that the $\mathbb{K}$-span of $\{e_{k}(\bar{\lambda},N) \mid \bar{\lambda} \in \mathcal{I}_{N,k}\}$ equals $(2\pi i)^{k}\mathcal{E}_{k}\big(\Gamma(N), \mathbb{K}\big)$. Therefore the collection $\{\mathfrak{e}_k(\bar{\lambda},N) \mid \bar{\lambda} \in \mathcal{I}_{N,k}\}$ is a $\mathbb{K}$-spanning set of $\mathcal{E}_{k}\big(\Gamma(N), \mathbb{K}\big)$.    

\begin{corollary}
\label{corollary4.5}
The homomorphism \eqref{4.6} maps $\mathcal{E}_k\big(\Gamma(N),\mathbb{K}\big)$ inside the canonical $\mathbb{K}$-form $H^{1}\big(\Gamma(N), V_{k-2,\mathbb{K}}\big)$.   
\end{corollary}

\begin{proof}
Proposition~\ref{proposition4.4new} shows that $[\text{ES}_k]\big(\mathfrak{e}_k(\bar{\lambda},N)\big) \in H^{1}\big(\Gamma(N), V_{k-2,\mathbb{Q}}\big)$ for each $\bar{\lambda} \in \mathcal{I}_{N,k}$. Thus, the assertion follows from the discussion above.     
\end{proof}

Suppose that $\mathcal{G}$ is a congruence subgroup containing the principal level $\Gamma(N)$. We compute the Eichler-Shimura homomorphism for $\mathcal{G}$ by restricting to $\Gamma(N)$ as in \eqref{2.3}. 

\begin{theorem}
\label{theorem4.6}
Let $\mathbb{K}$ be a subfield of $\mathbb{C}$ containing $\mathbb{Q}(\mu_N)$. Then the image of $\mathcal{E}_k(\mathcal{G}, \mathbb{K})$ under the Eichler-Shimura map  lies inside $H^{1}\big(\mathcal{G}, V_{k-2,\mathbb{K}}\big)$. 
\end{theorem}

\begin{proof}
We consider the commutative diagram 
\begin{equation*}
\begin{tikzcd}
\mathcal{M}_k(\mathcal{G}) \arrow[r, "{[\text{ES}_k]}"] \arrow[d]& H^{1}\big(\mathcal{G}, V_{k-2,\mathbb{C}}\big) \arrow[d]\\
\mathcal{M}_k\big(\Gamma(N)\big) \arrow[r, "{[\text{ES}_k]}"] & H^{1}\big(\Gamma(N), V_{k-2,\mathbb{C}}\big) 
\end{tikzcd}
\end{equation*}
where the vertical arrows are restriction maps. Let $f \in \mathcal{E}_k(\mathcal{G}, \mathbb{K})$. The compatibility for the rational structures of the space of Eisenstein series (Section~\ref{section3.1}) shows that the left vertical arrow maps $f$ inside $\mathcal{E}_k(\Gamma(N), \mathbb{K})$. Therefore, by Corollary~\ref{corollary4.5} the bottom arrow carries $f$ in $H^{1}\big(\Gamma(N), V_{k-2,\mathbb{K}}\big)$. Note that the isomorphism $H^{1}\big(\mathcal{G}, V_{k-2,\mathbb{C}}\big) \cong H^{1}\big(\Gamma(N), V_{k-2,\mathbb{C}}\big)^{\mathcal{G}}$ arising from restriction induces an isomorphism at the level of corresponding $\mathbb{K}$-forms: $H^{1}\big(\mathcal{G}, V_{k-2,\mathbb{K}}\big) \cong H^{1}\big(\Gamma(N), V_{k-2,\mathbb{K}}\big)^{\mathcal{G}}$. In particular, the preimage of $H^{1}\big(\Gamma(N), V_{k-2,\mathbb{K}}\big)$ under the right vertical arrow is precisely $H^{1}\big(\mathcal{G}, V_{k-2,\mathbb{K}}\big)$. Therefore, the image of $f$ under the top arrow must lie in $H^{1}\big(\mathcal{G}, V_{k-2,\mathbb{K}}\big)$. 
\end{proof}

\paragraph{Proof of Theorem~\ref{theorem1.1}.} Let the notation be as in the statement of the theorem. Suppose that $f \in \mathcal{B}_{\mathcal{G},k}$. It is clear that $f \in \mathcal{E}_k\big(\mathcal{G}, \mathbb{Q}\big)$. Letting $\mathbb{K} = \mathbb{Q}(\mu_N)$ in Theorem~\ref{theorem4.6}, we see that the image of $f$ under the Eichler-Shimura map is defined over $\mathbb{Q}(\mu_N)$ as desired. \hfill $\square$

Observe that the $\mathbb{K}$-subspace $H^{1}(\mathcal{G}, V_{k-2, \mathbb{K}})$ of $H^{1}(\mathcal{G}, V_{k-2, \mathbb{C}})$ is stable under the action of a double coset operator. Thus, one can use Theorem~\ref{theorem4.6} to conclude that the action of a double coset operator is definable on our $\mathbb{K}$-rational spaces of Eisenstein series; cf. Proposition~7.6 in \cite{myself}. 

\begin{corollary}
\label{corollary4.7}
Let the notation be as in Theorem~\ref{theorem4.6}. Then $\mathcal{E}_{k}(\mathcal{G}, \mathbb{K})$ is stable under the action of double coset operators. 
\end{corollary}  

\begin{proof}
Let $f \in \mathcal{E}_{k}(\mathcal{G}, \mathbb{K})$ and $T$ be a double coset operator attached to $\mathcal{G}$. Then \[T(f) \in \mathcal{E}_k(\mathcal{G}) \cap [\text{ES}_k]^{-1}\big(H^{1}(\mathcal{G}, V_{k-2,\mathbb{K}})\big) = \mathcal{E}_{k}(\mathcal{G}, \mathbb{K})\]
where the latter equality is a consequence of the theorem. 
\end{proof}

\begin{remark}
\label{remark4.8}
In this article, we compute the period polynomials for a special family of Eisenstein series on $\Gamma(N)$ and conclude the general statement using formal properties. Our method, however, can be enhanced to an algorithm to compute the periods of the spectral Eisenstein series on a general congruence subgroup in the induced picture; cf. \cite{myself}. The elliptic Eisenstein series are related to the $G$-series through explicit linear combinations, which enables us to write down the periods for $G$-series. Moreover, the spectral Eisenstein series on $\Gamma(N)$ are related to $G$-series via explicit identities with coefficients lying in $(2 \pi i)^k\mathbb{Q}(\mu_N)$ \cite[5.2]{myself}. Finally, a spectral Eisenstein series on $\mathcal{G}$ can be expressed as a sum of the spectral Eisenstein series on a principal level contained in it \cite[4.1]{myself}. However, we bypass the complicated computation and directly obtain rationality statements thanks to a robust formalism.  
\end{remark}

\section{An invariant of Zagier and Gangl}
\label{section5}
Let $K$ be an imaginary quadratic field and $\mathcal{O}_K$ be the ring of integers of $K$. With an ideal class $\mathfrak{r}$ of $\mathcal{O}_K$, one associates the partial zeta function $\zeta_{K, \mathfrak{r}}(s) = \sum_{\mathfrak{a} \in \mathfrak{r}} \frac{1}{N\mathfrak{a}^s}$. Zagier and Gangl \cite[7]{zagier-gangl} constructed a naturally defined invariant $\mathcal{I}_m(\mathfrak{r}) \in \mathbb{C}/\mathbb{Q}(m)$ so that 
\[\mathfrak{R}_m\big(\mathcal{I}_m(\mathfrak{r})\big) \sim_{\mathbb{Q}^{\times}} \frac{\lvert D_K \rvert^{m-\frac{1}{2}}}{\pi^m}\zeta_{K, \mathfrak{r}}(m)  \hspace{.3cm}(\substack{m \geq 2})\]
where $D_K$ is the discriminant of the extension $K/\mathbb{Q}$ and other notations are as in Section~\ref{section1.1}. Next, we use our generalization of the rational cocycle for the Eisenstein series to extend this construction to the Hecke $L$-functions.  

Let $\mathfrak{f}$ be an integral ideal of $\mathcal{O}_K$. Suppose that $I(\mathfrak{f})$ is the group of fractional ideals in $K$ coprime to $\mathfrak{f}$ and $P(\mathfrak{f})$ is the is the subgroup of principal ideals generated by an $\alpha \in K^{\times}$ so that $\alpha \equiv 1(\text{mod }\mathfrak{f})$. Now suppose $\chi: I(\mathfrak{f}) \to \mathbb{C}^{\times}$ is an algebraic Hecke character of type $(\delta,\delta)$ for some $\delta \in \mathbb{Z}$, that is to say, $\chi\big((\alpha)\big) = \lvert \alpha \rvert^{2\delta}$ for each $(\alpha) \in P(\mathfrak{f})$ \cite[Ch. 0]{schappacher}. If $\delta = 0$, then $\chi$ descends to a finite order Hecke character on the generalized class group $\mathfrak{C}_{\mathfrak{f}} = I(\mathfrak{f})/P(\mathfrak{f})$. The motive attached to a general algebraic Hecke character $\chi$ of type $(\delta,\delta)$ equals the $\delta$-th Tate twist of the Artin motive attached to $\chi$. At the level of $L$-functions, the Tate twist amounts to a shift in the argument. In more detail,  
\begin{equation}
\label{5.1}
\begin{aligned}
L(s, \chi) = \sum_{\substack{\mathfrak{a} \subseteq \mathcal{O}_K,\\ \gcd(\mathfrak{a}, \mathfrak{f}) = 1}} \frac{\chi(\mathfrak{a})}{N\mathfrak{a}^s} & = \sum_{\mathfrak{r} \in \mathfrak{C}_{\mathfrak{f}}} \sum_{\mathfrak{a} \in \mathfrak{r}} \frac{\chi(\mathfrak{a})}{N\mathfrak{a}^s} \hspace{1cm}\big(\substack{\text{Re}(s) > \delta+1}\big)\\
& = \frac{1}{w_{\mathfrak{f}}} \sum_{\mathfrak{r} \in \mathfrak{C}_{\mathfrak{f}}} \frac{N\mathfrak{b}_{\mathfrak{r}}^s}{\chi(\mathfrak{b}_{\mathfrak{r}})} \sum_{\eta \in \eta_{\mathfrak{r}}+ \mathfrak{b}_{\mathfrak{r}}\mathfrak{f}} \frac{1}{\lvert \eta \rvert^{2(s-\delta)}}
\end{aligned}
\end{equation}
where $w_{\mathfrak{f}}$ is the number of roots of unity in $\mathcal{O}_{K}^{\times}$ congruent to $1$ mod $\mathfrak{f}$, $\mathfrak{b}_{\mathfrak{r}} \in I(\mathfrak{f})$ is an integral ideal in the class $\mathfrak{r}^{-1}$, and $\eta_{\mathfrak{r}} \in \mathcal{O}_K$ is a solution to the system of congruences $\eta \equiv 1($\text{mod }$\mathfrak{f})$, $\eta \equiv 0($\text{mod }$\mathfrak{b}_{\mathfrak{r}})$ \cite[p.131]{lang}. We fix an embedding of $K$ into $\mathbb{C}$ and consider $\mathfrak{b}_{\mathfrak{r}}\mathfrak{f}$ as a lattice in $\mathbb{C}$. Let $N_{\mathfrak{r}}$ be the smallest positive integer so that $N_{\mathfrak{r}}\eta_{\mathfrak{r}} \in \mathfrak{b}_{\mathfrak{r}}\mathfrak{f}$. Write $\mathcal{L}_{\mathfrak{r}} := \frac{1}{N_{\mathfrak{r}}} \mathfrak{b}_{\mathfrak{r}}\mathfrak{f} \subseteq \mathbb{C}$. Observe that there exists an ordered $\mathbb{Z}$-basis $\{\omega_1(\mathfrak{r}), \omega_2(\mathfrak{r})\}$ of $\mathcal{L}_{\mathfrak{r}}$ with $\tau(\mathfrak{r}) = \frac{\omega_1(\mathfrak{r})}{\omega_2(\mathfrak{r})} \in \mathbb{H}$ and $\omega_2(\mathfrak{r}) \in \mathbb{Q}_{>0}$. Choose such a basis and write $\eta_{\mathfrak{r}} = \lambda_1(\mathfrak{r}) \omega_1(\mathfrak{r}) + \lambda_2(\mathfrak{r}) \omega_2(\mathfrak{r})$ with $\lambda_{\mathfrak{r}} = \big(\lambda_1(\mathfrak{r}), \lambda_2(\mathfrak{r})\big) \in \Lambda$. Now suppose $f_{\mathfrak{r}}(X) = aX^{2} + bX + c \in \mathbb{Z}[X]$ is the primitive minimal polynomial of $\tau(\mathfrak{r})$ with $a > 0$. Let $D_{\mathfrak{r}}$ be the discriminant of $f_{\mathfrak{r}}[X]$. The discriminant is related with the volume of $\mathcal{L}_{\mathfrak{r}}$, denoted $V_{\mathfrak{r}}$, through the identity $V_{\mathfrak{r}} = \omega_2(\mathfrak{r})^2 \text{Im}\big(\tau(\mathfrak{r})\big) =  \frac{\omega_2(\mathfrak{r})^2}{2a} \sqrt{\lvert D_{\mathfrak{r}}\rvert}$. 

Now, for an integer $m \geq2$, one can rewrite \eqref{5.1} as   
\begin{equation*}
\begin{aligned}
L(m + \delta, \chi) & = \frac{1}{w_{\mathfrak{f}}} \sum_{\mathfrak{r} \in \mathfrak{C}_{\mathfrak{f}}} \frac{N\mathfrak{b}_{\mathfrak{r}}^{m+\delta}}{\chi(\mathfrak{b}_{\mathfrak{r}})} \frac{1}{\omega_2(\mathfrak{r})^{2m}}\sum_{\substack{(c,d) \equiv \lambda_{\mathfrak{r}}(N_{\mathfrak{r}}), \\ (c,d) \neq (0,0)}} \frac{1}{\lvert c\tau(\mathfrak{r}) + d\rvert^{2m}},\\
& = \frac{1}{w_{\mathfrak{f}}N_{\mathfrak{r}}^2}\sum_{\mathfrak{r} \in \mathfrak{C}_{\mathfrak{f}}} \frac{N\mathfrak{b}_{\mathfrak{r}}^{m+\delta}}{\chi(\mathfrak{b}_{\mathfrak{r}})} \sum_{\bar{\lambda} \in \Lambda/N_{\mathfrak{r}}\Lambda} \mathbf{b}(\bar{\lambda}_{\mathfrak{r}}, \bar{\lambda}) \psi_{\mathfrak{r}}\big(m, \tau(\mathfrak{r}), \bar{\lambda}\big)  
\end{aligned}
\end{equation*}
where \[\psi_{\mathfrak{r}}\big(m, \tau(\mathfrak{r}), \bar{\lambda}\big) = \omega_2(\mathfrak{r})^{-2m}e_{m,m}\big(\tau(\mathfrak{r}), \bar{\lambda}, N_{\mathfrak{r}}\big) = \frac{\text{Im}\big(\tau(\mathfrak{r})\big)^m}{V_{\mathfrak{r}}^m}e_{m,m}\big(\tau(\mathfrak{r}), \bar{\lambda}, N_{\mathfrak{r}}\big)\] and $e_{m,m}(\cdot)$ is the nonholomorphic elliptic Eisenstein series described in \eqref{3.8new}. Note that $\psi_{\mathfrak{r}}$ depends on the choice of basis $\{\omega_1(\mathfrak{r}), \omega_2(\mathfrak{r})\}$ for $\mathcal{L}_{\mathfrak{r}}$ and replacing $\tau(\mathfrak{r})$ by $\gamma\tau(\mathfrak{r})$ gives rise to the transformation formula
\[\psi_{\mathfrak{r}}\big(m, \gamma\tau(\mathfrak{r}), \bar{\lambda}\gamma^{-1}\big) = \psi_{\mathfrak{r}}\big(m, \tau(\mathfrak{r}), \bar{\lambda}\big). \hspace{.3cm}(\substack{\gamma \in \Gamma})\]  
However, $\psi_{\mathfrak{r}}$ is determined by the constructions involving $(\mathfrak{b}_\mathfrak{r}, \mathfrak{f})$ and is independent of the data $\chi$. If $\mathfrak{f} = \mathcal{O}_K$ and $\chi$ is trivial, then $\psi_{\mathfrak{r}}\big(m, \tau(\mathfrak{r}), \bar{\lambda}\big)$ equals $\zeta_{K, \mathfrak{r}}(m)$ up to a nonzero rational multiple. Our result generalizing the construction of Zagier and Gangl is as follows:

\begin{theorem}
\label{theorem5.1}
Let the notation be as above. There exists a collection of $\mathbb{C}$-valued real analytic functions $\{\Psi_{\mathfrak{r}}\big(m, \tau, \bar{\lambda}\big) \mid \bar{\lambda} \in \Lambda/N_{\mathfrak{r}}\Lambda\}$ on $\mathbb{H}$ so that for each $\bar{\lambda} \in \Lambda/N_{\mathfrak{r}}\Lambda$ we have
\begin{enumerate}[label=(\roman*), align=left, leftmargin=0pt]
\item $\mathfrak{R}_m\big(\Psi_{\mathfrak{r}}(m, \tau(\mathfrak{r}), \bar{\lambda})\big) \sim_{\mathbb{Q}^{\times}} \frac{\lvert D_{\mathfrak{r}} \rvert^{m-\frac{1}{2}}}{\pi^m}\psi_{\mathfrak{r}}\big(m, \tau(\mathfrak{r}), \bar{\lambda}\big)$, 
\item $\Psi_{\mathfrak{r}}\big(m, \gamma\tau(\mathfrak{r}), \bar{\lambda}\gamma^{-1}\big) \equiv \Psi_{\mathfrak{r}}\big(m, \tau(\mathfrak{r}), \bar{\lambda}\big)$ in $\mathbb{C}/\mathbb{Q}(m)$, $\forall \gamma \in \Gamma$.  
\end{enumerate} 
\end{theorem} 

\begin{remark}
\label{remark5.2}
\begin{enumerate}[label=(\roman*), align=left, leftmargin=0pt]
\item We define the collection of invariants by lifting the elliptic Eisenstein series, rather than Hecke's $G$-series. These two series are related by linear combinations with coefficients in a cyclotomic field. Since the cyclotomic numbers generally do not behave well with the real and imaginary projections, there is no analog of the above statement for the $G$-series. Similarly, the method of this article does not readily extend to the algebraic Hecke characters with unequal indices. 
\item Traditionally, the polylogarithm conjecture is stated in the context of the special values of the Dedekind zeta function. Our work suggests, the polylogarithm conjecture for imaginary quadratic field should admit a generalization to Hecke $L$-functions in terms of $\psi_{\mathfrak{r}}\big(m, \tau(\mathfrak{r}), \bar{\lambda}\big)$. In more detail, we expect that each $\psi_{\mathfrak{r}}\big(m, \tau(\mathfrak{r}), \bar{\lambda}\big)$ can be expressed by the polylogarithm function evaluated at the elements of the $m$-th Bloch group of the ray class field attached to the modulus $\mathfrak{f}$.  
\end{enumerate}
\end{remark}

Our strategy for constructing the desired functions is to apply a specific differential operator to certain Eichler integrals. For an integer $k \in \mathbb{Z}$, let $D_{k}^{+} = k + (\tau - \bar{\tau})\partial_{\tau}$ denote the Maass raising operator. The Maass-Bol identity relates the Maass differential operator to the group action on the upper half plane:  
\[D_{k}^{+}\big((c\tau+d)^{-k}(c\bar{\tau}+d)^{-\ell}f(\gamma\tau)\big) = (c\tau+d)^{-k-1}(c\bar{\tau}+d)^{-\ell+1}D^{+}_k(f)(\gamma \tau)\]
where $f$ is a $\mathbb{C}$-valued $C^{\infty}$ function on $\mathbb{H}$, $\gamma \in \Gamma$, and $k, \ell \in \mathbb{Z}$. Set 
\[\mathbb{D}_{k, p} = D_{k}^{+} \circ D_{k -1}^{+} \circ \cdots \circ D_{k-p}^{+}. \hspace{.3cm}\substack{(p \geq 0)}\] 
Our construction requires the differential operator $\mathbb{D}_{m} := \mathbb{D}_{-m,m-2}$ where $m$ is an integer $\geq 2$.

\begin{lemma}
\label{lemma5.3} \emph{\cite[7]{zagier-gangl}}
Let the notation be as above. 
\begin{enumerate}[label=(\roman*), align=left, leftmargin=0pt]
\item $\mathbb{D}_m = (m-1)!\sum_{r = 0}^{m-1} \binom{-m}{r} \frac{(\tau - \bar{\tau})^{m-1-r} \partial_{\tau}^{m-1-r}}{(m-1-r)!}$.
\item $\mathbb{D}_{m}(\tau^{2m-1}) = (m-1)!\sum_{r = 0}^{m-1} (-1)^{r} \binom{2m-1}{r}\tau^{2m-1-r}\bar{\tau}^r$. 
\item Suppose that $0 \leq n \leq 2m-2$. Then 
\[\frac{\mathbb{D}_m(\tau^n)} {(\tau - \bar{\tau})^{m-1}} = Q (\frac{1}{\tau - \bar{\tau}}, \frac{\tau + \bar{\tau}}{\tau - \bar{\tau}}, \frac{\tau\bar{\tau}}{\tau - \bar{\tau}})\]  
where $Q \in \mathbb{Q}[X,Y,Z]$ is a homogeneous polynomial of degree $(m-1)$. 
\end{enumerate}
\end{lemma}

\begin{proof}
A straightforward inductive calculation verifies that 
\begin{equation*}
\mathbb{D}_{k, p} = \sum_{r = 0}^{p+1} r! \binom{k}{r}\binom{p+1}{p+1-r}(\tau - \bar{\tau})^{p+1-r} \partial_{\tau}^{p+1-r}. \hspace{.3cm}\substack{(p \geq 0)}
\end{equation*}
Part (i) is a special case of this identity. This expression for $\mathbb{D}_m$ is necessary to compute the action of the differential operator on the Fourier series, but does not directly yield the results involving polynomials in (ii) and (iii). To prove the latter parts, we perform explicit calculation using the identity
\begin{equation}
\label{5.2}
D_{k}^{+}(\tau^{n_1}\bar{\tau}^{n_2}) = (k+n_1)\tau^{n_1}\bar{\tau}^{n_2} - n_1\tau^{n_1-1}\bar{\tau}^{n_2+1}. \hspace{.3cm}(\substack{n_1, n_2, k \in \mathbb{Z}})
\end{equation}
Let $n \in \mathbb{Z}$ be fixed. For each $0 \leq p \leq m-2$ write \[\mathbb{D}_{2-2m+p,p}(\tau^n) = \sum_{r = 0}^{p+1} a^{(p+1)}_{r}\tau^{n-r}\bar{\tau}^r.\]
Suppose that $0 \leq p \leq m-3$. Then, an application of \eqref{5.2} yields the following system of equations:
\begin{align*}
a_{0}^{(p+2)} & = (3-2m+p+n)a_{0}^{(p+1)},\\
a_{r}^{(p+2)} & = (3-2m+p+n-r)a^{(p+1)}_{r} - (n-r+1)a_{r-1}^{(p+1)}, \hspace{.2cm} (\substack{1 \leq r \leq p+1})\\
a_{p+2}^{(p+2)} &= -(n-p-1)a_{p+1}^{(p+1)}. 
\end{align*}  
Note that $a^{(1)}_{0} = 2 - 2m + n$ and $a^{(1)}_{1} = -n$. Now, an easy calculation verifies that the sequence defined by 
\[a^{(p+1)}_{r} = (-1)^{p+1} (p+1)! \binom{n}{r} \binom{2m-2-n}{p+1-r} \hspace{.5cm} \big(\substack{0 \leq p \leq m-2, \; 0 \leq r \leq p+1}\big)\]
provides the unique solution to the recurrence above where $\binom{\bullet}{\bullet}$ is the formal binomial coefficient described in Section~\ref{section1.1}. Letting $p = m-2$ and $n = 2m-1$, we readily deduce the identity in part (ii). Suppose that $0 \leq n \leq m-1$. Then $a^{(m-1)}_{r}$ vanishes for each $n+ 1 \leq r \leq m-1$ and satisfies the relation \[a^{(m-1)}_{r} = a^{(m-1)}_{n-r}. \hspace{.3cm}(\substack{0 \leq r \leq n})\] Similarly, if $m \leq n \leq 2m-2$ then $a^{(m-1)}_{r}$ vanishes for each $0 \leq r \leq n-m$ and satisfies the relation  
\[a^{(m-1)}_{r} = a^{(m-1)}_{n-r}. \hspace{.3cm}(\substack{n-m+1 \leq r \leq m-1})\]
Therefore, for each $0 \leq n \leq 2m-2$ we have $\mathbb{D}_{m}(\tau^n) = P(\tau, \bar{\tau})$ where $P \in \mathbb{Q}[X,Y]$ is a homogeneous polynomial of degree $r$ that is invariant under the canonical action of $\mathbb{Z}/2\mathbb{Z}$ defined by switching $X$ and $Y$. Moreover, if $n \geq m$, then $P(X,Y) = (XY)^{n-m+1}P_1(X,Y)$ for some $P_1 \in \mathbb{Q}[X,Y]$. Observe that $X+Y$ and $XY$ generate the $\mathbb{Z}/2\mathbb{Z}$-invariant subalgebra of $\mathbb{Q}[X,Y]$. As a consequence, for each $0 \leq n \leq 2m-2$, there exists a homogeneous polynomial $Q \in \mathbb{Q}[X,Y,Z]$ with $\text{deg}(Q) = m-1$ so that  
\[\frac{\mathbb{D}_m(\tau^n)}{(\tau - \bar{\tau})^{m-1}} = Q(\frac{1}{\tau - \bar{\tau}}, \frac{\tau+ \bar{\tau}}{\tau - \bar{\tau}}, \frac{\tau\bar{\tau}}{\tau - \bar{\tau}}).\]  
\end{proof}

Recall the notion of the Eichler integral from Remark~\ref{remark2.3}. For simplicity, we abbreviate $\mathbb{I}_{k}(\tau, f, \tau)$ as $\mathbb{I}_k(\tau, f)$. Our proof of the theorem requires a formula for the Eichler integrals attached to the base point at infinity. 

\begin{lemma}
\label{lemma5.4}
Let $k$ be an integer $\geq 2$ and $f = \sum_{j = 0}^{\infty} A_j \mathbf{e}(\frac{j\tau}{N})$ be a modular form of weight $k$ on $\Gamma(N)$. Suppose that $\mathbb{I}_{k}^{\ast}(\tau, f, X)$ is the Eichler-Shimura integral with the basepoint at the boundary. Then, 
\[\mathbb{I}_{k}^{\ast}(\tau, f) = \frac{A_0\tau^{k-1}}{k-1} + (k-2)!\sum_{j = 1}^{\infty} \frac{A_jN^{k-1}}{(2\pi i j)^{k-1}}\mathbf{e}(\frac{j\tau}{N}).\]
\end{lemma}

\begin{proof}
Follows from the defining formula \eqref{2.4}. 
\end{proof}

\paragraph{Proof of Theorem~\ref{theorem5.1}.} Let the notation be as in the statement of the theorem. Define a choice of Eichler-Shimura integrals for $\mathcal{M}_{2m}\big(\Gamma(N_{\mathfrak{r}})\big)$ by \[\mathbb{I}_{2m}^{(1)}(\tau, f, X) = \mathbb{I}_{2m}^{\ast}(\tau, f, X) + i^{3-2m}L^{\ast}(f, 2m-1)\]
where $L^{\ast}$ is the completed $L$-function in \eqref{2.7}. Then the proof of Proposition~\ref{proposition4.4new} shows that for each $\bar{\lambda} \in \Lambda/N_{\mathfrak{r}}\Lambda$ and $\gamma \in \Gamma$ the difference 
\[P_{\gamma}(\bar{\lambda}, \tau) = (c\tau+d)^{2m-2}\mathbb{I}_{2m}^{(1)}\big(\gamma\tau, e_{2m}(\bar{\lambda}\gamma^{-1}, N_{\mathfrak{r}})\big) - \mathbb{I}_{2m}^{(1)}(\tau, e_{2m}\big(\bar{\lambda}, N_{\mathfrak{r}})\big),\] 
is a polynomial in $\tau$ of degree $2m-2$ with coefficients in $\mathbb{Q}(2m)$. Set 
\[\Psi_{\mathfrak{r}}(m, \tau, \bar{\lambda}) := \frac{i|D_{\mathfrak{r}}|^{\frac{m-1}{2}}}{\pi^{m}(\tau - \bar{\tau})^{m-1}}\mathbb{D}_{m}\big(\mathbb{I}_{2m}^{(1)}\big(\tau, e_{2m}(\bar{\lambda}, N_{\mathfrak{r}})\big)\big). \hspace{.3cm}(\substack{\bar{\lambda} \in \Lambda/N_{\mathfrak{r}}\Lambda})\]
It is clear that $\Psi_{\mathfrak{r}}(m, \tau, \bar{\lambda})$ defines a $\mathbb{C}$-valued real analytic function on $\mathbb{H}$. Now, an application of the Maass-Bol identity yields 
\[\Psi_{\mathfrak{r}}(m, \gamma\tau, \bar{\lambda}\gamma^{-1}) = \frac{i|D_{\mathfrak{r}}|^{\frac{m-1}{2}}}{\pi^{m}(\tau - \bar{\tau})^{m-1}}\mathbb{D}_{m}\big((c\tau+d)^{2m-2}\mathbb{I}_{2m}^{(1)}\big(\gamma\tau, e_{2m}(\bar{\lambda}\gamma^{-1}, N_{\mathfrak{r}})\big)\big).\]
Therefore, \[\Psi_{\mathfrak{r}}(m, \gamma\tau, \bar{\lambda}\gamma^{-1}) - \Psi_{\mathfrak{r}}(m, \tau, \bar{\lambda}) = \frac{i|D_{\mathfrak{r}}|^{\frac{m-1}{2}}}{\pi^{m}(\tau - \bar{\tau})^{m-1}} \mathbb{D}_{m}\big(P_{\gamma}(\bar{\lambda}, \tau)\big).\] Lemma~\ref{lemma5.3}(iii) shows that $(\tau - \bar{\tau})^{1-m}\mathbb{D}_{m}\big(P_{\gamma}(\bar{\lambda}, \tau)\big) = Q_{\gamma}(\frac{1}{\tau - \bar{\tau}}, \frac{\tau+ \bar{\tau}}{\tau - \bar{\tau}}, \frac{\tau\bar{\tau}}{\tau - \bar{\tau}})$ where $Q_{\gamma}$ is homogeneous polynomial in $\mathbb{C}[X,Y,Z]$ of degree $m-1$ with coefficients in $\mathbb{Q}(2m)$. Now suppose $\tau = \tau(\mathfrak{r})$. Then an easy calculation using the definition of $D_{\mathfrak{r}}$ verifies, \[\frac{i \lvert D_{\mathfrak{r}} \rvert^{\frac{m-1}{2}}} {\pi^{m}(\tau - \bar{\tau})^{m-1}}\mathbb{D}_m \big(P_{\gamma}(\bar{\lambda}, \tau)\big) \in \mathbb{Q}(m).\] In other words, the family of functions $\{\Psi_{\mathfrak{r}}(m, \tau, \bar{\lambda}) \mid \bar{\lambda} \in \Lambda/N_{\mathfrak{r}}\Lambda\}$ satisfies the second property in the statement of the theorem. To prove the first part, note that 
\[\mathfrak{R}_{m}(\alpha) = \begin{cases}
	\frac{\alpha + (-1)^{m-1}\bar{\alpha}}{2i}, & \text{if $m$ is even;}\\
	\frac{\alpha + (-1)^{m-1}\bar{\alpha}}{2}, & \text{if $m$ is odd.}
\end{cases} \hspace{.3cm}(\substack{\alpha \in \mathbb{C}})\]
Now, the identity in (i) follows from Lemma~\ref{lemma5.4}, Lemma~\ref{lemma5.3}, and the Fourier expansion formulas in Lemma~\ref{lemma3.1} and Corollary~\ref{corollary3.5}. \hfill $\square$

\paragraph{Acknowledgments.} The author wishes to thank Don Zagier for introducing him to the Eichler-Shimura theory and for lots of valuable advice. He is grateful to the Max Planck Institute for Mathematics, Bonn, for hospitality and financial support during the visit (2018-2019). The author is also indebted to the School of Mathematics at the Tata Institute of Fundamental Research for providing an exciting environment to study and work.

\end{document}